\documentclass[12pt,final]{amsart}
 
\usepackage{fullpage}
\usepackage[notcite,notref,color]{showkeys}

\usepackage{amsmath,amssymb,amsbsy,amsthm,bbm} 
\usepackage{enumitem}
\usepackage[overload]{textcase}
 
\usepackage{epsfig,psfrag,color,epstopdf} 
\usepackage{verbatim,xspace}
 
\usepackage[colorlinks,bookmarks,linkcolor=black,citecolor=black]{hyperref}

\newif\ifarxiv
\arxivtrue 

\setlist{itemsep=3pt,parsep=0pt,topsep=2pt,partopsep=0pt}  
\setlist{leftmargin=2.5\parindent} 

\def\itm#1{\rm ({#1})} \def\itmit#1{\itm{\it #1\,}} \def\rom{\itmit{\roman{*}}}
\def\abc{\itmit{\alph{*}}} 
 \def\endofClaim{\hfill\scalebox{.6}{$\Box$}}

\let\subset\subseteq \let\eps\varepsilon \let\rho\varrho
\def\dcup{\cup}

\def\cB{\mathcal{B}}  
\def\cF{\mathcal{F}} \def\cG{\mathcal{G}} 
  
\def\cS{\mathcal{S}} \def\cT{\mathcal{T}} 
\def\X{\mathcal{X}}

\def\le{\leqslant} \def\ge{\geqslant}

\newtheorem{theorem}{Theorem}[section]
\newtheorem*{SzRLdeg}{Szemer\'edi's Regularity Lemma (sparse minimum degree form)}

\newtheorem*{Chernoff}{Chernoff bound}
\newtheorem*{Hoeff}{Hoeffding's inequality}
\newtheorem{lemma}[theorem] {Lemma}    
\newtheorem{corollary}[theorem]{Corollary}    
\newtheorem{prop}[theorem] {Proposition}   
\newtheorem{fact}[theorem]{Fact} 
\newtheorem{obs}[theorem] {Observation}  
\newtheorem{conj}[theorem]{Conjecture}

\newtheorem{claim}{Claim}    
\newtheorem*{claim*}{Claim}  
\theoremstyle{definition}
\newtheorem{definition}[theorem]{Definition}
\newtheorem{defn}[theorem]{Definition}

\theoremstyle{remark} 

\newcommand{\oldqed}{}

\newenvironment{claimproof}[1][Proof]{
  \renewcommand{\oldqed}{\qedsymbol}
  \renewcommand{\qedsymbol}{\endofClaim}
  \begin{proof}[#1]
}{
  \end{proof}
  \renewcommand{\qedsymbol}{\oldqed}
}

\newcommand{\By}[2]{\overset{\mbox{\tiny{#1}}}{#2}} 
\newcommand{\ByRef}[2]{   \By{\eqref{#1}}{#2} }

\newcommand{\eqByRef}[1]{ \ByRef{#1}{=} }

\newcommand{\geByRef}[1]{ \ByRef{#1}{\ge} }

\newcommand{\NATS}{\mathbb{N}} 

\newcommand{\Prob}{\mathbb{P}}
\newcommand{\Exp}{\mathbb{E}}

 \def\Ex{\mathbb{E}}  \def\Pr{\mathbb{P}}
   \def\N{\mathbb{N}}

\newcommand{\ds}{\displaystyle}

\newcommand{\ol}{\overline}

\newcommand{\Bin}{\textup{Bin}}
\newcommand{\Hyp}{\textup{Hyp}}
\newcommand{\Aut}{\textup{Aut}}

\newcommand{\tpl}[2]{\mathbf{#1}^{#2}}

\title{Chromatic thresholds in dense random graphs}

  \author[P. Allen, J. B\"ottcher, S. Griffiths, Y. Kohayakawa \and R. Morris]{Peter Allen, Julia B\"ottcher, Simon Griffiths \\
 Yoshiharu Kohayakawa \and Robert Morris}
 \address{
    Peter Allen, Julia B\"ottcher \hfill\break Department of Mathematics, London School of Economics, Houghton Street, London WC2A 2AE, U.K.
 }
 \email{p.d.allen|j.boettcher@lse.ac.uk}
\address{
    Simon Griffiths \hfill\break
    Department of Statistics, University of Oxford, 1 South Parks Road,
    Oxford, OX1~3TG, UK.
} 
\email{griffith@stats.ox.ac.uk}
\address{
    Yoshiharu Kohayakawa \hfill\break Instituto de Matem\'atica e
    Estat\'{\i}stica, Universidade de S\~ao Paulo, Rua do Mat\~ao 1010,
    05508--090~S\~ao Paulo, Brazil.
  }
  \email{yoshi@ime.usp.br}
  \address{
    Robert Morris\hfill\break
   IMPA, Estrada Dona Castorina 110, Jardim Bot\^anico, Rio de Janeiro, RJ,
   Brazil
 }
 \email{rob@impa.br}

\thanks{
    PA was partially supported by FAPESP (Proc.~2010/09555-7); JB by FAPESP
    (Proc.~2009/17831-7); SG by CNPq (Proc.~500016/2010-2); YK by CNPq
    (Proc.~308509/2007-2); RM by CNPq (Proc.~479032/2012-2 and Proc.~303275/2013-8). This
    research was supported by CNPq (Proc.~484154/2010-9). The authors are
    grateful to NUMEC/USP, N\'ucleo de Modelagem Estoc\'astica e Complexidade of
    the University of S\~{a}o Paulo, and Project MaCLinC/USP, for supporting
    this research. }

\date{\today}

\begin{document}

\begin{abstract}
The chromatic threshold $\delta_\chi(H,p)$ of a graph $H$ with respect to the random graph $G(n,p)$ is the infimum over $d > 0$ such that the following holds with high probability: the family of $H$-free graphs $G \subset G(n,p)$ with minimum degree $\delta(G) \ge dpn$ has bounded chromatic number. The study of the parameter $\delta_\chi(H) := \delta_\chi(H,1)$ was initiated in 1973 by Erd\H{o}s and Simonovits, and was recently determined for all graphs $H$. In this paper we show that $\delta_\chi(H,p) = \delta_\chi(H)$ for all fixed $p \in (0,1)$, but that typically $\delta_\chi(H,p) \ne \delta_\chi(H)$ if $p = o(1)$. We also make significant progress towards determining $\delta_\chi(H,p)$ for all graphs $H$ in the range $p = n^{-o(1)}$. In sparser random graphs the problem is somewhat more complicated, and is studied in a separate paper. 
\end{abstract}

\maketitle

\section{Introduction}

One of the most famous early applications of the probabilistic method is
Erd\H{o}s' proof~\cite{Erd59} that there exist graphs with arbitrarily high
girth and chromatic number. In 1973, Erd\H{o}s and Simonovits~\cite{ES73}
asked whether such constructions are still possible under the additional
condition that the graph have high minimum degree. The \emph{chromatic
  threshold} $\delta_\chi(H)$ of a graph $H$ is defined to be the infimum
over $d > 0$ such that there exists $C = C(H,d)$ with the following
property: if $G$ is an $H$-free graph on $n$ vertices with minimum degree
$\delta(G) \ge d n$, then $\chi(G) \le C$. For example, it is easy to see
that $\delta_\chi(H) = 0$ for all bipartite $H$, and it was proved by
Thomassen~\cite{Thomassen02,Thomassen07} that $\delta_\chi(K_3) = 1/3$ and
that $\delta_\chi(C_{2k+1}) = 0$ for every $k \ge 2$. 

Important breakthroughs in the study of chromatic thresholds of more general families of graphs were
obtained by Lyle~\cite{Lyle10} and {\L}uczak and Thomass\'e~\cite{LT}. Following~\cite{LT}, we say that a 
graph $H$ is \emph{near-acyclic} if $\chi(H)=3$ and~$H$ admits a partition
into a forest~$F$ and an independent set~$I$ such that every odd cycle
of~$H$ meets~$I$ in at least two vertices. The family of near-acyclic
graphs was introduced by \L uczak and Thomass\'e~\cite{LT}, who conjectured
that they were exactly the $3$-chromatic graphs with chromatic threshold
zero. This was proved in~\cite{ChromThresh}, where moreover the chromatic
threshold of every graph $H$ was determined:
If $\chi(H) = r \ge 3$ then
\begin{equation*}
\delta_\chi(H) \, \in \, \bigg\{ \frac{r-3}{r-2}, \, \frac{2r-5}{2r-3}, \, \frac{r-2}{r-1} \bigg\}\,,
\end{equation*}
where the first possibility occurs exactly when it is possible to obtain a
near-acyclic graph by removing $r-3$ independent sets from $H$, the third
when $H$ has no forest in its decomposition family\footnote{Recall that the
decomposition family of a graph $H$ is the collection of bipartite
graphs obtained from $H$ by removing $\chi(H) - 2$ independent sets.}, and the
second otherwise.

In recent years, beginning with~\cite{FR,KLR}, there has been a great deal of interest in \emph{sparse random analogues} of results in extremal graph theory. For example, it can be proved\footnote{See~\cite{Sam} for a proof of this theorem using the method of~\cite{SchTuran}.} using the general methods introduced in~\cite{BMS,ConGow,ST,SchTuran} that if $p \gg n^{-2/(r+1)}$ then, with high probability, every $K_{r+1}$-free subgraph of $G(n,p)$ with $\big( 1 - 1/r  + o(1) \big) p {\binom{n}{2}}$ edges can be made $r$-partite by removing $o(pn^2)$ edges. DeMarco and Kahn~\cite{DeMKah1,DeMKah2} moreover proved that if $p \gg n^{-2/(r+1)} (\log n)^{2/(r+1)(r - 2)}$ then with high probability the largest $K_{r+1}$-free subgraph of $G(n,p)$ is $r$-partite. For an excellent introduction to the area, see the recent survey~\cite{RS}. 

In this paper we will study a sparse random analogue of the chromatic threshold. The following definition was first made in~\cite{ChromThresh}.

\begin{defn}\label{def:deltachip}
Given a graph $H$ and a function $p = p(n) \in [0,1]$, define
\begin{align*}
& \delta_\chi\big( H, p \big) \, := \, \inf \Big\{ d > 0 \,:\, \text{there exists $C > 0$ such that the following holds} \\ 
& \hspace{5cm} \text{with high probability: every $H$-free spanning subgraph }  \\
& \hspace{6.5cm} G \subset G(n,p) \textup{ with } \delta(G) \ge d pn \textup{ satisfies } \chi(G) \le C \Big\}.
\end{align*}
We call $\delta_\chi\big(H,p\big)$ the \emph{chromatic threshold of $H$ with respect to $p$}.
\end{defn}

Note that $\delta_\chi(H) = \delta_\chi(H,1)$, so this definition generalises that of the chromatic threshold. We emphasise that the constant $C$ is allowed to depend on the graph $H$, the function $p$ and the number $d$, but not on the integer $n$. We also note that if, for some $d$, with high probability there is no spanning $H$-free subgraph of $G(n,p)$ whose average degree exceeds $dpn$, then vacuously we have $\delta_\chi(H,p)\le d$.

\subsection{Our results}

Our first theorem shows that if $p \in (0,1]$ is constant, then the chromatic threshold does not depend on its value.

\begin{theorem}\label{thm:pconst}
  For each constant $p > 0$ and graph $H$, we have
  $\delta_\chi(H,p) = \delta_\chi(H)$.
\end{theorem}

For functions $p = o(1)$, the situation is quite different. In this paper we will focus on the `dense' regime, by which we mean the case in which $p = n^{-o(1)}$. In this regime, it is still trivially true that $\delta_\chi(H,p) = \delta_\chi(H) = 0$ for all bipartite $H$. We are also able to determine $\delta_\chi(H,p)$ in the case $\chi(H) \ge 4$, even for somewhat smaller values of $p$. Recall that the $2$-density $m_2(H)$ is the maximum of $\frac{e(F) - 1}{v(F) - 2}$ over subgraphs $F \subset H$ with at least $3$ vertices.

\begin{theorem}\label{thm:classhigh} 
Let $H$ be a graph with $\chi(H) \ge 4$, and let $p = p(n)$ be any function satisfying $\max\big\{ n^{-1/m_2(H)}, n^{-1/2} \big\} \ll p = o(1)$. Then
\[\delta_{\chi}(H,p) \,=\, \frac{\chi(H)-2}{\chi(H)-1}\,.\]
\end{theorem}

Recall that $\delta_{\chi}(H) = \frac{\chi(H)-2}{\chi(H)-1}$ if and only if there is no forest in the decomposition family of $H$, so for all other graphs we have $\delta_\chi(H,p) > \delta_\chi(H)$ whenever $p \to 0$ sufficiently slowly. We remark that the upper bound in Theorem~\ref{thm:classhigh} is an immediate consequence of a theorem of Conlon and Gowers~\cite{ConGow} and Schacht~\cite{SchTuran}, while the lower bound is proved in~\cite{sparse}. 

Perhaps surprisingly, for graphs $H$ with chromatic number $\chi(H) = 3$ the situation is significantly more complicated. In order to state our main theorem, which determines $\delta_{\chi}(H,p)$ for many (but not all) 3-chromatic graphs in the `dense' range $p = n^{-o(1)}$, we will need the following definitions (see Figure~\ref{fig:thunder}). 

\begin{defn}\label{def:cftcf}
A graph $H$ is a \emph{cloud-forest graph} if there is an independent set $I \subset V(H)$ (the \emph{cloud}) such that $V(H)\setminus I$ induces a forest $F$, the only edges from $I$ to $F$ go to leaves or isolated vertices of $F$, and no two adjacent leaves in $F$ send edges to $I$. 

Moreover, $H$ is a \emph{thundercloud-forest graph} if there is a cloud $I \subset V(H)$, which witnesses that $H$ is a cloud-forest graph, such that every odd cycle in $H$ uses at least two vertices of~$I$. 
\end{defn}

\begin{figure}[ht]
\psfrag{I}{$I$}
\psfrag{J}{$J$}
\includegraphics[width=12cm]{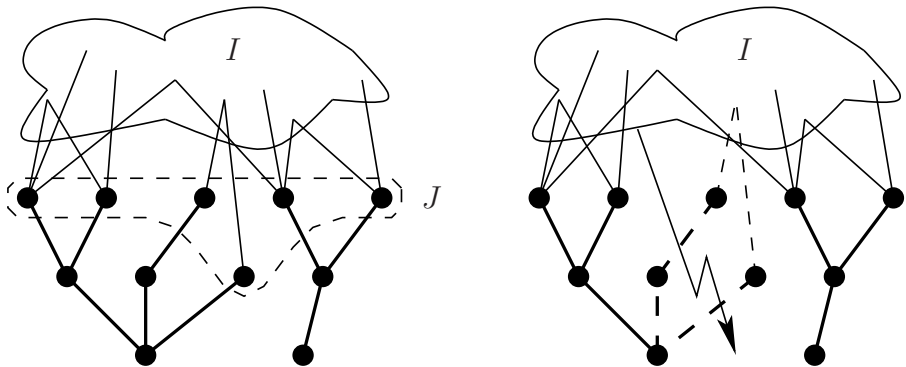}
\label{fig:thunder}
\caption{A cloud-forest graph (left), and a forbidden odd cycle for a thundercloud-forest graph (dashed, right).}
\end{figure}

An alternative definition of `cloud-forest', which is used in our proofs, is the following: The vertex set
can be partitioned into independent sets $I$ and $J$, and a forest $F'$,
such that there are no edges from $V(F')$ to $I$ and each vertex of $J$ has
at most one neighbour in $V(F')$. To obtain this partition from that of Definition~\ref{def:cftcf}, let $J$ be the neighbours of vertices in $I$, as shown in Figure~\ref{fig:thunder}, and let $F'$ be obtained by removing $J$ from $F$.

For example, $K_3$ is not a cloud-forest graph, $C_5$ is a cloud-forest but not
a thundercloud-forest graph, and $C_{2k+1}$ is a thundercloud-forest graph
for every $k \ge 3$. Note that every cloud-forest graph has a forest in its decomposition family, and similarly every thundercloud-forest graphs is near-acyclic, but in both cases the reverse inclusion does not hold (as $K_3$ and $C_5$ respectively demonstrate). 

Our main theorem is the following partial characterisation of $\delta_\chi(H,p)$ for 3-chromatic graphs in the dense regime, i.e., for functions $p = p(n)$ that satisfy $p = o(1)$ and $p = n^{-o(1)}$. Together with Theorem~\ref{thm:classhigh}, it determines $\delta_{\chi}(H,p)$ in this range for every $H$ that is not a thundercloud forest graph.

\begin{theorem}\label{thm:sparseclass} 
Let $H$ be a graph with $\chi(H) = 3$, and let $p = p(n)$ be a function satisfying $p = o(1)$ and $p = n^{-o(1)}$. Then
\begin{equation*}
\delta_{\chi}(H,p)\begin{cases} \, = \, \tfrac12&\text{ if $H$ is not a cloud-forest graph.}\\
\, = \, \tfrac13 & \text{ if $H$ is a cloud-forest graph but not a thundercloud-forest graph.}\\
\, \le \, \tfrac13 & \text{ if $H$ is a thundercloud-forest graph.}
\end{cases}
\end{equation*}
\end{theorem}

We consider this theorem (in particular, the upper bound $\delta_{\chi}(H,p) \le 1/3$ for cloud-forest graphs) to be the main contribution of this paper. It is possible that $\delta_{\chi}(H,p) = 1/3$ for every function $\max\big\{ n^{-1/m_2(H)}, n^{-1/2} \big\} \ll p \ll 1$ whenever $H$ is a cloud-forest graph but not a thundercloud-forest graph, see~\cite[Question~6.3]{sparse}.

For thundercloud-forest graphs Theorem~\ref{thm:sparseclass} only provides an upper bound on $\delta_\chi(H,p)$, which we do not believe to be sharp. We make the following conjecture, which would complete the characterisation of $\delta_{\chi}(H,p)$ in dense random graphs.

\begin{conj}\label{conj:thunder} 
If $H$ is a thundercloud-forest graph, and $p = p(n)$ is a function satisfying $p = o(1)$ and $p = n^{-o(1)}$, then $\delta_\chi(H,p)=0$.
\end{conj}

In~\cite[Theorem~1.5]{sparse}, we prove Conjecture~\ref{conj:thunder} for odd cycles, that is, we prove that $\delta_\chi(C_{2k+1},p)=0$ for every $k \ge 3$.

\subsection{Sparser random graphs} 

In a companion paper~\cite{sparse} we study $\delta_\chi(H,p)$ for sparser random graphs, i.e., when $p=p(n)$  tends to zero faster than $n^{-\eps}$ for some $\eps > 0$. For example, in that paper we determine $\delta_\chi(H,p)$ for almost all values of $p$ whenever $\chi(H) \ge 5$: 
\begin{equation}\label{eq:thm:classhigh}
\delta_{\chi}(H,p) = \left\{
\begin{array}{cll}
\delta_\chi(H) & \text{if } & p > 0 \text{ is constant,}\smallskip\\ 
\tfrac{\chi(H)-2}{\chi(H)-1} & \text{if } & n^{-1/m_2(H)} \ll p \ll 1,\smallskip\\
1 & \text{if } & \frac{\log n}{n} \ll p \ll n^{-1/m_2(H)}.
\end{array} \right.
\end{equation}
Note that if $p \ll \frac{\log n}{n}$ then $G(n,p)$ is likely to have an isolated vertex, so trivially $\delta_\chi(H,p) = 0$. In the range $p = \Theta\big( n^{-1/m_2(H)} \big)$ we are not sure exactly what to expect, see~\cite[Problem~6.4]{sparse}.

As noted above, the situation is more complicated (and more interesting) for 3-chromatic graphs. In~\cite{sparse} we are able to determine $\delta_\chi(K_3,p)$ and $\delta_\chi(C_5,p)$ for most functions $p$, and $\delta_\chi(C_{2k+1},p)$ if either $p \gg n^{-1/2}$ or $\frac{\log n}{n} \ll p \ll n^{-(2k-3)/(2k-2)}$, for all $k \ge 3$. Perhaps most interestingly, we show that $\delta_\chi(C_5,p)$ has (at least) four different non-trivial regimes: 
\begin{equation*}
\delta_{\chi}(C_5,p) = \left\{
\begin{array}{cll}
0 & \text{if } & p > 0 \text{ is constant}\smallskip\\ 
\tfrac{1}{3} & \text{if } & n^{-1/2} \ll p \ll 1\smallskip\\
\tfrac{1}{2} & \text{if } & n^{-3/4} \ll p \ll n^{-1/2}\smallskip\\
1 & \text{if } & \frac{\log n}{n} \ll p \ll n^{-3/4}.
\end{array} \right.
\end{equation*}
We also show that~\eqref{eq:thm:classhigh} holds for a large class of 4-chromatic graphs (those with $m_2(H) > 2$), but we suspect (see~\cite[Conjecture~6.1]{sparse}) that this is not the case for all 4-chromatic graphs.

All of the lower bound constructions in the range $p = o(1)$ are given in~\cite{sparse}, but in Section~\ref{sec:constructconst} we will briefly describe those that are used in the proofs of Theorems~\ref{thm:classhigh} and~\ref{thm:sparseclass}. As noted above, the main contribution of this paper is the proof (see Section~\ref{sec:slightly}) that $\delta_{\chi}(H,p) \le 1/3$ if $H$ is a cloud-forest graph and $p = p(n)$ satisfies $p = o(1)$ and $p = n^{-o(1)}$.

\subsection{An approximate version of \texorpdfstring{$\delta_\chi(H,p)$}{delta(H)}} \label{sec:intro:CGSS}

In the definition of $\delta_{\chi}(H,p)$, the $H$-free graph $G \subset G(n,p)$ is required to be $C$-partite (rather than `close-to-$C$-partite'), and in this sense the theorems stated above have more in common with the theorem of DeMarco and Kahn~\cite{DeMKah1,DeMKah2}, stated earlier, than those of Conlon and Gowers~\cite{ConGow} and Schacht~\cite{SchTuran} (see the discussion before Definition~\ref{def:deltachip}). The following `approximate' random graph version of $\delta_\chi(H)$ was recently proposed by Conlon, Gowers, Samotij and Schacht~\cite{CGSS}.

\begin{defn}\label{def:deltastar}
For each graph $H$ and function $p = p(n) \in (0,1]$, define
\begin{align*}
& \delta_\chi^*\big( H, p \big) := \inf \Big\{ d > 0 \,:\, \text{there exists $C > 0$ such that, for all $\eps > 0$, the following holds} \\ 
& \hspace{3.5cm} \text{with high probability: every $H$-free spanning subgraph $G \subset G(n,p)$}  \\
& \hspace{3.5cm}  \textup{with } \delta(G) \ge d pn \textup{ can be made $C$-colourable by removing $\eps p n^2$ edges} \Big\}.
\end{align*}
\end{defn}

Conlon, Gowers, Samotij and Schacht~\cite{CGSS} used the so-called Kohayakawa-{\L}uczak-R\"odl conjecture~\cite{KLR}, which was recently proved in~\cite{BMS,CGSS,ST}, to deduce the following theorem, which determines $\delta_\chi^*(H,p)$ in terms of $\delta_\chi^*(H,1)$ for all $H$ and essentially all $p$. 

\begin{theorem}[Conlon, Gowers, Samotij and Schacht~\cite{CGSS}]\label{thm:CGSS} 
For every graph $H$, 
\begin{equation*}
\delta_{\chi}^*(H,p) = \left\{
\begin{array}{cll}
\delta_{\chi}^*(H) & \text{if } & p \gg n^{-1/m_2(H)}\smallskip\\
1 & \text{if } & \frac{\log n}{n} \ll p \ll n^{-1/m_2(H)},
\end{array} \right.
\end{equation*}
where $\delta_{\chi}^*(H) := \delta_{\chi}^*(H,1)$. 
\end{theorem}

It is somewhat surprising that $\delta_{\chi}(H,p)$ and $\delta_{\chi}^*(H,p)$ have such different behaviour. Indeed, the threshold for an exact statement (such as that of DeMarco and Kahn) `usually' differs from that for the asymptotic statement (such as that proved by Conlon-Gowers and Schacht) by only a poly-logarithmic factor. By contrast, for most graphs $H$ there are at least two thresholds at which the value of $\delta_\chi(H,p)$ changes, and in the case $H=C_5$ there are (at least) three such thresholds. It seems reasonable to believe that multiple thresholds also exist for many other cloud-forest graphs.

Perhaps surprisingly, we do not have $\delta_\chi^*(H) = \delta_\chi(H)$ in general. However, we are able to determine $\delta_\chi^*(H)$ in terms of $\delta_\chi(H)$ for every graph $H$.

\begin{theorem}\label{thm:deltachistar}
For every graph $H$, 
\begin{equation*}
 \delta_\chi^*(H) = \min\Big\{\delta_\chi(H')\,:\, \text{there exists a
   homomorphism from }H\text{ to }H' \Big\}\,.
\end{equation*}
\end{theorem}

For example if $H$ is obtained from $C_5$ by blowing up each vertex to a $2$-vertex independent set then we have $\delta_\chi(H)=\tfrac12$ but $\delta_\chi^*(H)=0$. On the other hand, for some graphs $H$, such as $K_3$, we have $\delta_\chi^*(H)=\delta_\chi(H)$. 

Finally, let us note that is interesting to ask how many edges one really needs to delete in order
to obtain a graph with bounded chromatic number: perhaps one can replace
$\eps pn^2$ in the definition of~$\delta_\chi^*(H,p)$ by some
asymptotically smaller function of $n$ and $p$? For the specific case of
$K_3$, Allen, B\"ottcher, Kohayakawa and Roberts~\cite{ABKR} have shown that
$\eps p n^2$ can be replaced by any function $f(n,p) \gg n/p$, but there
exists $c > 0$ such that the function $cn/p$ does not suffice.

\subsection{The structure of the paper}

In Section~\ref{prelim:sec} we state the extremal and probabilistic tools we will use in the proofs of Theorems~\ref{thm:pconst} and~\ref{thm:sparseclass}, together with the sparse random Erd\H{o}s-Stone Theorem, which implies some of the upper bounds for Theorems~\ref{thm:pconst} and~\ref{thm:sparseclass}, and the upper bound of Theorem~\ref{thm:classhigh}. In Section~\ref{sec:constructconst} we describe the constructions that prove the lower bounds in Theorems~\ref{thm:pconst},~\ref{thm:classhigh} and~\ref{thm:sparseclass}, and in Section~\ref{sec:slightly} we prove the upper bound on $\delta_\chi(H,p)$ for cloud-forest graphs, which is our main new result and which completes the proof of Theorem~\ref{thm:sparseclass}. In Section~\ref{sec:pconst} we show how to adapt the method of~\cite{ChromThresh} in order to prove that $\delta_\chi(H,p) \le \delta_\chi(H)$ for all fixed $p > 0$, and hence complete the proof of Theorem~\ref{thm:pconst}. Finally, in Section~\ref{sec:deltachistar}, we give a brief sketch of the proof of Theorem~\ref{thm:deltachistar}.

\section{Preliminaries}\label{prelim:sec}

In this section we state the sparse random Erd\H{o}s-Stone theorem, which implies some of our claimed upper bounds, some basic probabilistic and graph-theoretic tools, and a sparse version of Szemer\'edi's Regularity Lemma.

\subsection{The sparse random \texorpdfstring{Erd\H{o}s}{Erdos}-Stone theorem}

The following theorem was originally conjectured by Kohayakawa, \L uczak and R\"odl~\cite{KLR}, and was recently proved by Conlon and Gowers~\cite{ConGow} (for strictly balanced graphs $H$) and Schacht~\cite{SchTuran} (in general), see also~\cite{BMS,ST}. 

\begin{theorem}\label{thm:CGS}
For every graph $H$, every $\gamma > 0$, and every $p \gg n^{- 1 / m_2(H)}$, the following holds with high probability. For every $H$-free subgraph $G \subset G(n,p)$, we have
$$e(G) \, \le \, \bigg( 1 - \frac{1}{\chi(H)-1} + \gamma \bigg) p {\binom{n}{2}}.$$
\end{theorem}

Theorem~\ref{thm:CGS} has the following immediate corollary. 

\begin{corollary}\label{cor:turan}
For every graph $H$, and every $p \gg n^{- 1 / m_2(H)}$, 
$$\delta_\chi(H,p) \,\le\, 1 - \frac{1}{\chi(H)-1}.$$
\end{corollary}

\begin{proof}
Let $\gamma > 0$, and suppose that $G \subset G(n,p)$ is a spanning subgraph with $\delta(G) \ge \big( 1 - \frac{1}{\chi(H)-1} + \gamma \big) p n$. Since $p \gg n^{- 1 / m_2(H)}$, it follows from Theorem~\ref{thm:CGS} that, with high probability, $H \subset G$. Hence $\delta_\chi(H,p) \le 1 - \frac{1}{\chi(H)-1}$, as claimed.
\end{proof}

\subsection{Probabilistic and graph-theoretic tools}

We will frequently use the following concentration bounds. Let $\Bin(n,p)$ denote the Binomial distribution, and let $\Hyp(n,m,s)$ denote the hypergeometric distribution, corresponding to respectively the number of elements of $[n]$ chosen if each is selected independently with probability $p$, and the number of elements of $[m]$ chosen if a set of $s$ elements of $[n]$ is selected uniformly at random. The following relatively weak bounds on the large deviations of $\Bin(n,p)$ and $\Hyp(n,m,s)$, see for example~\cite[Theorems~2.1 and~2.10]{JLRbook}, will suffice for our purposes.

\begin{Chernoff}
Let $n \in \N$ and $p \in [0,1]$, and let $X \sim \Bin(n,p)$. Then
$$\Pr\big( | X - \Ex[X]| \ge t \big) \, \le \, e^{-\Omega(t)}$$
for every $t = \Omega\big( \Ex[X] \big)$.
\end{Chernoff}

\begin{Hoeff}
Let $n \in \N$ and $m,s \in [n]$, and let $Y \sim \Hyp(n,m,s)$. Then
$$\Pr\big( | Y - \Ex[Y]| \ge t \big) \, \le \, e^{-\Omega(t)}$$
for every $t = \Omega\big( \Ex[Y] \big)$.
\end{Hoeff}

We will also need the following supersaturation theorem of Erd\H{o}s and Simonovits~\cite{ErdSimSup}.

\begin{theorem}[Erd\H{o}s and Simonovits]\label{lem:densbip} 
For every $s \in \N$ there exists $\beta > 0$ such that the following holds. Let $G$ be a graph on $n$ vertices, with $e(G) = \rho n^2 \ge \beta^{-1}n^{2-1/s}$ edges, where $\rho=\rho(n)$. Then $G$ contains at least $\beta \rho^{s^2}n^{2s}$ copies of $K_{s,s}$.
\end{theorem}

Finally, we will use the following straightforward and well-known fact.

\begin{fact}\label{prop:forest}
Let~$F$ be a forest and~$G$ be a graph on $n$ vertices. If $e(G) \ge v(F) n$, then $F\subset G$.
\end{fact}

\subsection{Sparse regularity in \texorpdfstring{$G(n,p)$}{G(n,p)}}

One of our key tools in this paper will be the so-called `sparse minimum degree form' of Szemer\'edi's Regularity Lemma. In order to state this result we need a little notation.

\begin{defn} [$(\eps,p)$-regular pairs and partitions, the reduced graph]
Given a graph $G$ and $\eps, d, p > 0$, a pair of disjoint vertex sets $(A,B)$ is said to be \emph{$(\eps,p)$-regular} if 
\[\bigg| \frac{e\big( G[A,B] \big)}{p|A| |B|} - \frac{e\big( G[X,Y] \big)}{p|X| |Y|} \bigg| \, < \, \eps \]
for every $X \subset A$ and $Y \subset B$ with $|X| \ge \eps |A|$ and $|Y| \ge \eps |B|$.
We say that the pair $(A,B)$ is \emph{$(\eps,d,p)$-lower-regular} if $e\big( G[X,Y] \big) \ge (d-\eps)p|X||Y|$ for every $X \subset A$ and $Y \subset B$ with $|X| \ge \eps |A|$ and $|Y| \ge \eps |B|$. Finally, we say that $(A,B)$ is \emph{$(\eps,d,p)$-regular} if it is both $(\eps,p)$-regular and $(\eps,d,p)$-lower-regular.

A partition $V(G) = V_0 \cup \ldots \cup V_k$ is said to be \emph{$(\eps,p)$-regular} if $|V_0| \le \eps n$, $|V_1| = \ldots = |V_k|$, and at most $\eps k^2$ of the pairs $(V_i,V_j)$ with $1 \le i < j \le k$ are not $(\eps,p)$-regular. 

The \emph{$(\eps,d,p)$-reduced graph} of an $(\eps,p)$-regular partition is the graph $R$ with vertex set  
$V(R) = \{1,\ldots,k\}$ and edge set 
$$E(R) \, = \, \big\{ ij : (V_i,V_j) \text{ is an $(\eps,d,p)$-regular pair} \big\}.$$ 
\end{defn}

We will use the following form of the Regularity Lemma, see~\cite[Lemma~4.4]{BipartBU} for a proof. Note that although the statement there only guarantees lower-regularity of pairs in the partition, the proof explicitly gives an $(\eps,p)$-regular partition.

\begin{SzRLdeg}
Let $\delta,d,\eps > 0$, $k_0 \in \N$ and $p = p(n) \gg (\log n)^4 / n$. There exists $k_1 = k_1(\delta,d,\eps,k_0) \in \N$ such that the following holds with high probability. If $G \subset G(n,p)$ has minimum degree $\delta(G) \ge \delta pn$, then there is an $(\eps,p)$-regular partition of $G$ into $k$ parts, where $k_0\le k\le k_1$, whose $(\eps,d,p)$-reduced graph $R$ has minimum degree at least $(\delta-d-\eps)k$.
\end{SzRLdeg}

When $p$ is constant, we will use an associated `Counting Lemma'  (see~\cite[Theorem~3.1]{KS93}, for example) which says that if $H \subset R$ then $G$ contains a positive fraction of all copies of $H$. 

\begin{lemma}[Counting Lemma]
Given $p > 0$ and $d > 0$, if $\eps > 0$ is sufficiently small the following holds for each $r_1\in\N$ when $n \in \N$ is sufficiently large. If $G$ is a graph on~$n$ vertices with $(\eps,d,p)$-reduced graph~$R$ on $r\le r_1$ vertices, such that $H \subset R$, then~$G$ contains at least $\tfrac{1}{2v(H)!}(dp)^{e(H)}(n/r)^{v(H)}$ copies of~$H$. 
\end{lemma}

When $p = o(1)$, the proof of the Counting Lemma breaks down, since large subsets of an $(\eps,d,p)$-regular pair no longer necessarily have sufficiently strong regularity properties. However, the following theorem shows that this desired `inheritance of regularity' holds, with high probability, for almost all neighbourhoods in $G(n,p)$. This result follows with some work from Gerke, Kohayakawa, Steger and R\"odl~\cite{GKRS}, for the details see~\cite{SparseBU}.

\begin{theorem}\label{thm:2reg}
 For any $\eps',d > 0$ there exist $\eps_0 = \eps_0(\eps',d) > 0$ and $C = C(\eps',d)$ such that, for any $0 < \eps \le \eps_0$ and any function $p = p(n)$, the following holds with high probability for the graph $\Gamma = G(n,p)$. For any disjoint sets of vertices $X$ and $Y$ with 
 $$\min\big\{ |X|,|Y| \big\} \, \ge \, C \max\bigg\{ \frac{1}{p^2}, \frac{\log n}{p} \bigg\},$$ 
 and any $(\eps,d,p)$-lower-regular bipartite subgraph $G \subset \Gamma[X,Y]$, there are at most 
 $$C \max\bigg\{ \frac{1}{p^2}, \frac{\log n}{p} \bigg\},$$ 
 vertices $w \in V(\Gamma)$ such that $\big( N_\Gamma(w) \cap X, N_\Gamma(w) \cap Y \big)$ is not $(\eps',d,p)$-lower-regular in $G$.
\end{theorem}

Finally, let us state a useful (and straightforward) fact about regular pairs, known as the `Slicing Lemma'.

\begin{lemma}[Slicing Lemma]\label{lem:slicing}
Let $(U,W)$ be an $(\eps,d,p)$-regular (respectively, lower-regular) pair and $U'\subset U$, $W'\subset W$ satisfy $|U'|\ge \alpha|U|$ and $|W'|\ge\alpha|W|$. Then $(U',W')$ is $(\eps/\alpha,d,p)$-regular (respectively, lower-regular).\qed 
\end{lemma}

\section{The lower bounds: constructions}\label{sec:constructconst}

The aim of this short section is to state or adapt constructions from~\cite{sparse,ChromThresh,LT} which imply the lower bounds in Theorem~\ref{thm:pconst},~\ref{thm:classhigh} and~\ref{thm:sparseclass}. We first state the lower bound construction from~\cite{sparse} which, together with Corollary~\ref{cor:turan}, proves Theorem~\ref{thm:classhigh}.

\begin{theorem}[{\cite[Theorem 3.5]{sparse}}]\label{thm:classhigh:lower}
Let $H$ be a graph with $\chi(H) \ge 4$, and suppose $p=p(n)$ satisfies $n^{-1/2} \ll p\ll 1$. Then
\[\delta_\chi(H,p) \, \ge \, 1 - \frac{1}{\chi(H)-1}\,.\]
\end{theorem}

The next two constructions together give the lower bounds for Theorem~\ref{thm:sparseclass}.

\begin{prop}[{\cite[Proposition 4.2]{sparse}}]\label{prop:cloud:lower}
Let $H$ be a graph with $\chi(H) = 3$, and suppose that $n^{-1/2} \ll p(n) = o(1)$. If $H$ is not a cloud-forest graph, then
$$\delta_{\chi}(H,p) \ge \frac{1}{2}.$$
\end{prop}

\begin{prop}[{\cite[Proposition 4.3]{sparse}}]\label{prop:thundercloud:lower}
Let $H$ be a graph with $\chi(H) = 3$, and suppose that $n^{-1/2} \ll p(n) = o(1)$. If $H$ is not a thundercloud-forest graph, then
$$\delta_{\chi}(H,p) \ge \frac{1}{3}.$$
\end{prop}

All three of these constructions use the fact, proved in~\cite[Proposition~2.1]{sparse}, that if $n^{-1/2} \ll p = o(1)$, then $G(n,p)$ contains (with high probability) a subgraph $F$ with $o(1/p)$ vertices, and arbitrarily high chromatic number and girth. This is not easy to prove for polynomially sparse random graphs: but it is worth noting that in the dense regime, when $p=n^{-o(1)}$, we can use the fact that for each constant $t$, with high probability $G(n,p)$ contains $K_t$, and appeal to Erd\H{o}s' result~\cite{Erd59} that there exist graphs with arbitrarily large girth and chromatic number to obtain $F$ much more easily.

The first two constructions take the subgraph $F$, and (with high probability) find a partition of the remaining vertices of $G(n,p)$ into $\chi(H)-1$ roughly equal parts $V_1,\ldots,V_{\chi(H)-1}$, such that each vertex of $F$ has about $\big(1-\tfrac{1}{\chi(H)-1}\big)pn$ neighbours in $V_1$. For Theorem~\ref{thm:classhigh:lower} we can then let $G$ consist of $F$, together with the edges from $F$ to $V_1$, and all edges between $V_i$ and $V_j$ for $i\neq j$. It is easy to check that this $G$ has minimum degree close to $\big(1-\tfrac{1}{\chi(H)-1}\big)pn$ and large chromatic number, while any $v(H)$-vertex subgraph of $G$ intersects $F$ in a forest, so that we can colour the vertices in each part $V_i$ with colour $i$ and the forest in $F$ with colours $2$ and $3$. This gives a proper $\big(\chi(H)-1\big)$-colouring, so in particular $H\not\subset G$.

For Proposition~\ref{prop:cloud:lower} we modify this construction slightly, removing edges so that no two vertices of $F$ have a common neighbour in $V_1$ in $G$. One can easily verify that this extra deletion does not significantly decrease the minimum degree of $G$, and again it is easy to check that any $v(H)$-vertex subgraph of $G$ is a cloud-forest graph (the cloud consists of the vertices in $V_2$).

Finally, the proof of Proposition~\ref{prop:thundercloud:lower} is similar, but more complicated. In addition to the ideas above, it makes use of a construction of {\L}uczak and Thomass\'e~\cite{LT}, that was originally used to show $\delta_\chi(H)\ge\tfrac13$ for graphs $H$ that are not near-acyclic. For the details, and for proofs of all three statements, we refer the reader to~\cite{sparse}.

\medskip

It remains to prove the lower bound in Theorem~\ref{thm:pconst}, which is an immediate consequence of the following proposition.

\begin{prop}\label{prop:pconstant:construction}
Fix a graph $H$ and a constant $p > 0$. Then $\delta_\chi(H,p) \ge \delta_\chi(H)$. That is, for each $C > 0$ and $\gamma > 0$, with high probability $G(n,p)$ contains a spanning $H$-free subgraph with minimum degree at least $\big(\delta_\chi(H) - \gamma\big)pn$ and chromatic number at least $C$.
\end{prop}

Proposition~\ref{prop:pconstant:construction} follows by adapting the constructions from~\cite{ChromThresh}, which themselves are minor adaptations of constructions originally due to {\L}uczak and Thomass\'e~\cite{LT}. We state only the features of these constructions we require in order to prove Proposition~\ref{prop:pconstant:construction}, and refer the reader to~\cite{ChromThresh,LT} for further details. The following lemma was proved in~\cite{ChromThresh} as Proposition~5, Theorem~16 or Proposition~35, depending on the value of $\delta_\chi(H)$.

\begin{lemma}\label{lem:constconst}
 For every graph $H$ and constants $C,\gamma > 0$, there exists $K = K(H,\gamma,C) > 0$ such that the following holds. For all sufficiently large $n$, there exists an $H$-free graph $G$ on $n$ vertices with the following properties:
 \begin{enumerate}[label=\abc]
\item $\delta(G) \ge \big(\delta_\chi(H)-\gamma\big)n$.
\item There exist disjoint sets $X,Y \subset V(G)$, with $|X|=K$ and $|Y|=n/K$, such that
$$\chi\big(G[X]\big) \ge C \qquad \text{and} \qquad e\big( G[X,Y] \big) = e\big( G[Y] \big) = 0.$$
\end{enumerate}
\end{lemma}

Note that for the application in~\cite{ChromThresh} the important points are the high minimum degree and the subgraph $G[X]$ whose chromatic number is large. However for this paper we require in addition the existence of the set $Y$. It is easy to check that each of the constructions in~\cite{ChromThresh} permits us to find such a set. For the constructions given there as Propositions~5 and~35, we need to choose in the construction not just any `Erd\H{o}s graph', that is, a graph with chromatic number at least $C$ and girth at least $v(H)+1$, but specifically one in which all but a fixed number $K$ of vertices are independent; this is possible since graphs with chromatic number $C$ and girth $v(H)+1$ exist. We then let $X$ be the $K$ vertices which are not independent. For the construction given as Theorem~16, we always obtain the desired sets, taking $X$ to be the vertex set of the Borsuk subgraph in that construction.

The next lemma allows us to find in $G(n,p)$ a graph corresponding to that described in Lemma~\ref{lem:constconst}.

\begin{lemma}\label{lem:constcsn} 
Let $d,p,\gamma,K > 0$, and let $G$ be a graph on $n$ vertices satisfying:
 \begin{enumerate}[label=\abc]
\item $\delta(G) \ge dn$.
\item There exist disjoint sets $X,Y \subset V(G)$, with $|X|=K$ and $|Y|=n/K$, such that
$$e\big( G[X,Y] \big) = e\big( G[Y] \big) = 0.$$
\end{enumerate}
Then, with high probability, $G(n,p)$ contains as a spanning subgraph a subgraph of $G$ with minimum degree at least $(d-\gamma)pn$ which includes all edges of $G[X]$.
\end{lemma}

\begin{proof}
We expose the edges of $G(n,p)$ in two rounds. First, we expose all the edges among the first $K+n/K$ vertices. It is easy to see that, with high probability, we will find a $K$-vertex clique in this set. Fix any injective map $\phi \colon V(G) \to [n]$ which takes $X$ to the $K$-vertex clique and $Y$ to the remaining $n/K$ initial (i.e., already exposed) vertices. 

We next expose the remaining edges, and claim that, with high probability,
the intersection of $\phi(G)$ and $G(n,p)$ has minimum degree at least
$(d-\gamma)pn$. Indeed, since $\delta(G) \ge dn$ and $d$, $p$ and $\gamma$
are fixed, this follows easily by Chernoff's inequality and the union bound. Thus, by construction, we have found the desired subgraph of $G$.
\end{proof}

Proposition~\ref{prop:pconstant:construction} now follows immediately.

\begin{proof}[Proof of Proposition~\ref{prop:pconstant:construction}]
Given $H$, $C$, $p$ and $\gamma$, let $n$ be sufficiently large, and let $K = K(H,\gamma,C) > 0$ and $G$ be given by Lemma~\ref{lem:constconst}, so in particular $G$ is $H$-free. Now, applying Lemma~\ref{lem:constcsn} with $d = \delta_\chi(H) - \gamma$, it follows that, with high probability, $G(n,p)$ contains a spanning subgraph $G' \subset G$ with minimum degree at least $\big( \delta_\chi(H) - 2\gamma \big)pn$ which includes all edges of $G[X]$, so $\chi(G') \ge C$. Since $\gamma > 0$ was arbitrary, this proves the proposition. 
\end{proof}

\section{The upper bound for cloud-forest graphs}\label{sec:slightly}

In this section we will prove the following proposition, which is our main new result.

\begin{prop}\label{prop:slightly} 
If $H$ is a cloud-forest graph, and $p = p(n)$ satisfies $p = o(1)$ and $p = n^{-o(1)}$, then 
\[\delta_\chi(H,p) \le \, \frac{1}{3}\,.\]
\end{prop}

This proposition, together with Corollary~\ref{cor:turan} and Propositions~\ref{prop:cloud:lower} and~\ref{prop:thundercloud:lower}, proves Theorem~\ref{thm:sparseclass}. We begin by giving  an outline of the proof of Proposition~\ref{prop:slightly}.

Let us fix a cloud-forest graph $H$ with $s$ vertices, and a function $p = p(n)$ such that $p = o(1)$ and $p = n^{-o(1)}$. Fix also a (small) constant $\gamma > 0$ and (with foresight) set $d = \gamma/6$. We will use the following definitions.

\begin{defn}
We define the \emph{$(d,p)$-robust second neighbourhood} $N^*_2(v)$ of a vertex $v$ in a graph $G$ to be the set of vertices $w$ of $G$ such that $v$ and $w$ have at least $dp^2n$ common neighbours in $G$. (Since $d$ and $p$ were fixed above, we suppress them from the notation.)

Given $u,v \in V(G)$, let us say that a set $Z$ of size $s$ is \emph{$(u,v)$-completable} if $Z$ has at least $s$ common neighbours in each of $N(u)$ and $N(v)$.
\end{defn}

Let $G$ be an $H$-free spanning subgraph of $G(n,p)$ with $\delta(G) \ge \big( \frac{1}{3} + 2\gamma \big) pn$. Applying the sparse minimum degree form of Szemer\'edi's Regularity Lemma to $G$, we obtain a partition $V(G) = V_0 \cup V_1 \cup \cdots \cup V_k$, such that the $(\eps,d,p)$-reduced graph $R$ satisfies $\delta(R)\ge\big(\tfrac13+\gamma\big)k$.
For each $i\in[k]$, we let $X_i$ be the set of vertices $v\in V(G)$ such that $N^*_2(v)$ covers at least a $\big(\tfrac12+\gamma\big)$-fraction of $V_i$, and set $X_0:=V(G)\setminus\big(X_1\cup\dots\cup X_k\big)$. We will show that $\chi\big( G[X_i] \big) = O(1)$ for each $1 \le i \le k$, and that $X_0=\emptyset$. 

We will bound the chromatic number of $G[X_i]$ in three steps, as follows. First, we show (see Claim~\ref{clm:second}, below) that for every pair $u,v \in X_i$, there are $\Omega(n^s)$ sets $Z \subset V_i$ of size $s$ that are $(u,v)$-completable. Second (see Claim~\ref{clm:third}), we show that if $\chi\big( G[X_i] \big) \gg 1$, then there exists a subgraph $E' \subset G[X_i]$ with arbitrarily large minimum degree. Using the pigeonhole principle, we can show that there exists a subgraph $E \subset E'$ with large average degree and an $s$-set $Z \subseteq V_i$ such that $Z$ is $(u,v)$-completable for each $uv \in E$. Third (see Claim~\ref{clm:fourth}), a graph with large enough average degree contains all small forests, by Fact~\ref{prop:forest}. Using the alternative definition of  a cloud-forest graph, it follows that there exists a forest $F'$ that is contained in $E$ and which we can extend to a copy of $H$. 

In order to complete the proof, we will show that $X_0 = \emptyset$, as follows. Suppose for a contradiction that there exists a vertex $u \in X_0$. Since the vertices of $X_0$ spread out their second neighbourhoods relatively evenly over the clusters of the regular partition, we will be able to find (see Claim~\ref{clm:fifth}) a pair $(V_i,V_j)$ of clusters which form an $(\eps,d,p)$-regular pair in $G$, and furthermore are such that there are at least $2dpn / 3$ vertices in $N(u)$ with at least $dp|V_i| / 3$ neighbours in each of $V_i$ and $V_j$. We will then (see Claim~\ref{clm:sixth}) use Theorem~\ref{thm:2reg} to show that for at least $dpn / 3$ of those vertices $v \in N(u)$, the density of $\big( N(v) \cap V_i, N(v) \cap V_j\big)$ is at least $dp / 2$ (this is `inherited' from the $(\eps,d,p)$-regular pair $(V_i,V_j)$). Now, by Theorem~\ref{lem:densbip}, any such bipartite graph contains many copies of $K_{s,s}$, which together with $v$ gives us many copies of $K_{1,s,s}$. By an application of the pigeonhole principle we find a copy of $K_{s,s,s}$ in $G$. Since $H \subset K_{s,s,s}$, we thus obtain the desired contradiction.

In order to perform the first step (Claim~\ref{clm:second}) in the above sketch, we need the following lemma. It implies that if $(W_1,Y)$ and $(W_2,Y)$ are two `sufficiently dense' pairs in a subgraph $G \subset G(n,p)$, and $p$ is not too small, then we can find `many' copies of $K_{s,2s}$ with the smaller part in $Y$, and the other part split equally between $W_1$ and $W_2$.

\begin{lemma}\label{lem:makebip}
For every $\eps > 0$ and $s \in \N$, there exists $\alpha > 0$ such that the following holds for all sufficiently large $n \in \N$. Let $G$ be a graph on $n$ vertices, and let $p = n^{-o(1)}$. Let $W_1,W_2, Y \subset V(G)$ be disjoint sets with $|Y| \ge 2pn$, with $\eps pn \le |W_i| \le 2pn$, and with
\[|N(y) \cap W_i| \ge \eps p^2 n \qquad \text{and} \qquad \Big| \bigcap_{v \in S} N(v) \cap W_i \Big| \le 2p^{s+1}n\]
for each $i \in \{1,2\}$, $y \in Y$ and $S \subset Y$ of size $s$. Then there exist at least $\alpha |Y|^s$ sets $Z \subset Y$ of size $s$ such that 
$$\Big| \bigcap_{v \in Z} N(v) \cap W_i \Big| \ge s$$
for each $i \in \{1,2\}$. 
\end{lemma}

\begin{proof}
Fix $\eps > 0$ and $s \in \N$, let $\beta > 0$ be the constant returned by Theorem~\ref{lem:densbip}, and set
\[
\alpha \,:=\, \ds\frac{\beta}{2} \bigg( \frac{\eps}{6s} \bigg)^s \bigg( \ds\frac{\eps^2}{32} \bigg)^{s^2}.
\]
Let us assume without loss of generality that $|W_1| \ge |W_2|$, and say that a copy of $K_{s,2s}$ in the bipartite graph with parts $Y$ and $W_1 \cup W_2$ is \emph{balanced} if it has $s$ vertices in each of $W_1$ and $W_2$. Our aim is to show that at least $\alpha |Y|^s$ sets $Z\subseteq Y$ of size $s$ extend to a balanced copy of $K_{s,2s}$. We will do so in two steps: we will prove a lower bound on the total number of balanced copies of $K_{s,2s}$, and an upper bound on the number rooted at a given $s$-set in $Y$. By the pigeonhole principle, this will be enough to give the result.

To prove a lower bound on the number of balanced copies of $K_{s,2s}$, we will apply Theorem~\ref{lem:densbip} to the following random bipartite graph $\mathbf{H}$. Let $\phi \colon W_2 \to W_1$ be a uniformly chosen random injective map, and let $Y' \subset Y$ be a uniformly chosen random subset of size $2pn$. Let $\mathbf{H}$ be the bipartite graph with parts $Y'$ and $W_1$, and edge set 
$$E(\mathbf{H}) = \big\{ yw \in G[Y',W_1] \,:\, y\phi^{-1}(w)\in G[Y',W_2] \big\}.$$ 
(If $\phi^{-1}(w)$ is not defined, we say the pair $y\phi^{-1}(w)$ is not in $G$.) Note that each copy of $K_{s,s}$ in $\mathbf{H}$ corresponds to a balanced copy of $K_{s,2s}$. 

We claim first that, with high probability, 
\begin{equation}\label{eq:eHbound}
e(\mathbf{H}) \, \ge \, \frac{\eps^2 p^4 n^2}{2} \, \ge \, \frac{\eps^2p^2}{32} \cdot v(\mathbf{H})^2.
\end{equation}
To prove this, let $y \in Y'$, and recall that $y$ has at least $\eps p^2 n$ neighbours in each of $W_1$ and $W_2$, and that $|W_2| \le |W_1| \le 2pn$. Since $\phi$ is a random map, the degree $d_\mathbf{H}(y)$ of $y$ in $\mathbf{H}$ is hypergeometrically distributed with mean at least $(\eps p / 2 ) \cdot \eps p^2 n$, and it follows (by Hoeffding's inequality, and the fact that $p = n^{-o(1)}$) that
$$\Pr\bigg( d_\mathbf{H}(y) \le \frac{\eps^2 p^3 n}{4} \bigg) \, \ll \, \frac{1}{n}.$$ 
By the union bound, and recalling that $|Y'| = 2pn$, the claimed bound on $e(\mathbf{H})$ follows.

Let $\mathbf{K}$ denote the number of copies of $K_{s,s}$ in $\mathbf{H}$. We claim that 
\begin{equation}\label{eq:ExK}
\Exp[ \mathbf{K} ] \, \ge \, \ds\frac{\beta}{2} \bigg( \ds\frac{\eps^2p^2}{32} \bigg)^{s^2} (2pn)^{2s}.
\end{equation}
To prove this, observe that if~\eqref{eq:eHbound} holds then Theorem~\ref{lem:densbip}, applied with $\rho=\eps^2p^2/32$, implies that $\mathbf{K}\ge\beta\rho^{s^2}v(H)^{2s}$. Since~\eqref{eq:eHbound} holds with probability greater than $1/2$, and $v(H)\ge|Y'|= 2pn$, the bound~\eqref{eq:ExK} follows immediately.

We will next show that the number of balanced copies of $K_{s,2s}$ is at least
\begin{equation}\label{eq:makebip:numks2s}
\bigg( \frac{|Y|}{3pn} \bigg)^s \bigg( \frac{\eps pn}{2s} \bigg)^s \cdot \Exp [\mathbf{K}].
\end{equation}
To see this, simply note that for a given balanced copy $K$ of $K_{s,2s}$, the $s$-set is contained in the randomly chosen $Y'$  with probability
\[\binom{|Y| - s}{2pn - s} \binom{|Y|}{2pn}^{-1} \, \le \, \bigg( \frac{2pn}{|Y| - s} \bigg)^s  \, \le \, \bigg( \frac{3pn}{|Y|} \bigg)^s\,.\]
If this event occurs, then $K$ yields a copy of $K_{s,s}$ in $\mathbf{H}$ with probability
\[\frac{s!\binom{|W_1|-s}{|W_2|-s}\big(|W_2|-s\big)!}{\binom{|W_1|}{|W_2|}|W_2|!}=
\frac{s! (|W_1| - s)!}{|W_1|!} \, \le \, \frac{s!}{(\eps pn - s)^s} \, \le \, \bigg( \frac{2s}{\eps pn} \bigg)^s\,.\] 
where the first inequality uses the condition $|W_1|\ge\eps p n$ of the lemma, and the second uses the facts that $p=n^{-o(1)}$ and $n$ is sufficiently large. Thus $\Exp [\mathbf{K}]$ is at most the product of these two probabilities, times the number of choices for $K$, as claimed in~\eqref{eq:makebip:numks2s}.

Finally observe that, since no $s$-set of vertices of $Y$ has more than $2p^{s+1}n$ neighbours in either $W_1$ or $W_2,$ each $s$-set of vertices of $Y$ extends to at most $(2p^{s+1}n)^{2s}$ balanced copies of $K_{s,2s}$. It follows from~\eqref{eq:ExK} and~\eqref{eq:makebip:numks2s} that the number of $s$-sets in $Y$ which extend to at least one balanced copy of $K_{s,2s}$ is at least
\[\bigg( \frac{|Y|}{3pn} \bigg)^s \bigg( \frac{\eps pn}{2s} \bigg)^s \cdot \ds\frac{\beta}{2} \bigg( \ds\frac{\eps^2p^2}{32} \bigg)^{s^2} (2pn)^{2s} \cdot (2p^{s+1}n)^{-2s} \, = \, \alpha |Y|^s\,,\]
as required.
\end{proof}

Before proving Proposition~\ref{prop:slightly}, let us note several properties of $\Gamma=G(n,p)$ that hold with high probability if $p=n^{-o(1)}$. In the proof we will assume that all of these properties hold:
\begin{enumerate}[label=(A\arabic*)]
\item\label{DNF:nbh} For each $|S| = O(1)$ we have $\Big| \ds\bigcap_{u \in S} N_{\Gamma}(u) \Big| = \big( 1+o(1) \big) p^{|S|} n$.
\item\label{DNF:edgeint} For each $|U| = \Omega(pn)$ there are at most $p|U|^2$ edges in $U$, and there are at most $\frac{\log n}{p^2}$ vertices outside $U$ with more than $2p|U|$ neighbours in $U$.
\item\label{DNF:edgecross} For any disjoint sets $U$ and $V$ of size $\Omega( p n )$ there are $\big( 1+o(1) \big) p|U||V|$ edges between $U$ and $V$.
\end{enumerate}
In each case, the probability of failure can easily be shown to tend to zero using the Chernoff bound. 

\begin{proof}[Proof of Proposition~\ref{prop:slightly}]
Let $H$ be a cloud-forest graph with $s$ vertices, and let $\gamma > 0$. We claim that there exists $C=C(H,\gamma)$ such that the following holds with high probability: if $G$ is a $H$-free spanning subgraph of $G(n,p)$ with $\delta(G) \ge \big(\tfrac13+2\gamma\big)pn$, then $\chi(G)\le C$. Let $k_0 = 1 / \gamma$, let $d = \gamma / 6$, and let $\eps' > 0$ be sufficiently small. Let $\eps_0 < \eps'$ and $C'$ be the constants returned by Theorem~\ref{thm:2reg} with inputs $\eps'$ and $d$, and choose $\eps \le \eps_0/4$ sufficiently small.

Suppose that $\Gamma=G(n,p)$ satisfies assumptions~\ref{DNF:nbh},~\ref{DNF:edgeint} and~\ref{DNF:edgecross}, and the high probability events of Theorem~\ref{thm:2reg} and of the sparse minimum degree form of Szemer\'edi's Regularity Lemma. Let $G\subset\Gamma$ be an $H$-free graph with minimum degree $\big(\tfrac13+2\gamma\big)pn$. We will write $N(u)$ for the neighbourhood of $u$ in $G$, and $N_\Gamma(u)$ for the neighbourhood in $\Gamma$.

We begin by applying the sparse minimum degree form of Szemer\'edi's Regularity Lemma, with input $\delta=\tfrac13+2\gamma$, $d$, $\eps$ and $k_0$, to $G$. This gives us a partition of $V(G)$ into parts $V_0,V_1,\ldots,V_k$, where $k_0\le k\le k_1$ and $k_1=k_1(\delta,d,\eps,k_0)$ does not depend on $n$, and a reduced graph $R$ with
\begin{equation}\label{eq:deltaR}
\delta(R) \ge \bigg( \frac{1}{3} + \gamma \bigg) k\,.
\end{equation}
We now define sets $X_1,\ldots,X_k$ by
\begin{equation}\label{def:Xi:NR}
X_i := \bigg\{ v \in V(G) : \big| N^*_2(v) \cap V_i \big| \ge \bigg( \frac{1}{2} + d \bigg) |V_i| \bigg\}
\end{equation}
and let $X_0=V(G)\setminus\big(X_1\cup\dots\cup X_k\big)$. Let $\alpha' > 0$ be the constant returned by Lemma~\ref{lem:makebip} with input $\eps > 0$ and $s$, and set $\alpha = d^s \alpha'$.

Our first goal is to show that for each $i$ we have $\chi\big(G[X_i]\big)=O(1)$. We break the proof up into a series of claims, the first of which gives us the set $W_1$ that we will use in our application of Lemma~\ref{lem:makebip}.

\begin{claim}\label{clm:first} 
For every $u \in V(G)$, there exists $W_1 \subset N(u)$ of size $d^2pn / 6$ such that 
\begin{equation}\label{eq:C11}
|N(w)\cap W_1| \ge d^3p^2n / 24
\end{equation}
for every $w \in N^*_2(u)$, and 
\begin{equation}\label{eq:C12}
|N(w) \cap W_1|  \le 2d^2p^2 n
\end{equation}
 for every $w \in V(G)$.
\end{claim}

\begin{claimproof}[Proof of Claim~\ref{clm:first}]
Choose a set $W_1 \subset N(u)$ of size $d^2pn / 6$ uniformly at random. By the definition of $N^*_2(u)$, we have $|N(u) \cap N(w)| \ge d p^2 n$ for every $w \in N^*_2(u)$. Moreover, by assumption~\ref{DNF:nbh}, and our bound on $\delta(G)$, we have $pn/3 \le |N(u)| \le 2 p n$ and $|N(u) \cap N(w)| \le 2p^2 n$ for every $w\in V(G)$. It follows that $|N(w)\cap W_1|$ is a hypergeometrically distributed random variable with expected value at least $d^3 p^2 n/12$ for every $w\in N^*_2(u)$, and at most $d^2p^2 n$ for every $w \in V(G)$. 

Since $p = n^{-o(1)}$, Hoeffding's inequality and the union bound tell us that with high probability $|N(w)\cap W_1|$ satisfies~\eqref{eq:C11} for every $w \in N^*_2(u)$, and~\eqref{eq:C12} for every $w\in V(G)$. Thus there must exist some such set $W_1 \subset N(u)$, as claimed.
\end{claimproof}

We now show that for each $i$ and pair $u,v\in X_i$, there are many sets in $V_i$ which are $(u,v)$-completable. 

\begin{claim}\label{clm:second} 
For every $i \in [k]$, and every pair $u,v \in X_i$, there exist at least $\alpha |V_i|^s$ subsets $Z \subset V_i$ of size $s$ that are $(u,v)$-completable. 
\end{claim}

\begin{claimproof}[Proof of Claim~\ref{clm:second}]
Let $W_1 \subset N(u)$ be given by Claim~\ref{clm:first}, and set
\[W_2 = N(v) \setminus W_1 \qquad \text{and} \qquad Y = N^*_2(u) \cap N^*_2(v) \cap V_i \setminus \big( W_1 \cup W_2 \cup \{u,v\} \big)\,.\]
We will use Lemma~\ref{lem:makebip} to show that there exist at least $\alpha' |Y|^s$ sets $Z \subset Y$ of size $s$ whose common neighbourhoods intersect each of $W_1$ and $W_2$ in at least $s$ vertices, and which are therefore $(u,v)$-completable.

In order to apply Lemma~\ref{lem:makebip}, we need to check that $W_1$, $W_2$ and $Y$ satisfy the various conditions of the lemma. To do so, note first that 
\[|W_1| = \frac{d^2pn}{6} \qquad \textup{and} \qquad \bigg( \frac{1}{3} - \frac{d^2}{6} \bigg) pn \le |W_2| \le |N(v)| \le 2pn\]
by Claim~\ref{clm:first}, our lower bound on $\delta(G)$, and assumption~\ref{DNF:nbh}. Thus we have $\eps pn \le |W_i| \le 2pn$ for $i \in \{1,2\}$, as required.

Next, to bound $|Y|$, note that since $u,v \in X_i$, by inclusion-exclusion and~\eqref{def:Xi:NR} we have 
\[\big| N^*_2(u) \cap N^*_2(v)  \cap V_i \big| \, \ge \, 2 \bigg( \frac12 + d \bigg) |V_i| - |V_i| \, =\, 2d|V_i| \,.\]
Since $|W_1| + |W_2|+2 \le 4pn+2\le d|V_i|$ and $p = o(1)$, we obtain $|Y| \ge d|V_i| \ge 2pn$.

To bound $|N(y) \cap W_i|$, we use Claim~\ref{clm:first} and the fact that $Y \subset N^*_2(u) \cap N^*_2(v)$. Indeed, by~\eqref{eq:C11} we have $|N(y) \cap W_1| \ge d^3 p^2 n / 24 \ge \eps p^2 n$ for every $y \in Y \subset N^*_2(u)$, and by~\eqref{eq:C12}, together with the fact $W_2=N(v)\setminus W_1$, we have
$$|N(y) \cap W_2| \ge |N(y) \cap N(v)| - |N(y) \cap W_1| \ge d p^2 n - 2d^2p^2 n \, \ge \,  \eps p^2 n$$ 
for every $y \in Y \subset N^*_2(v)$. Finally, since $W_1\subset N(u)$, $W_2 \subset N(v)$ and $u,v \notin Y$, it follows from assumption~\ref{DNF:nbh} that 
$$\bigg| \bigcap_{w \in S} N(w) \cap W_i \bigg| \le 2p^{s+1}n$$ 
for $i \in \{1,2\}$ and every set $S \subset Y$ of size $s$, as required. 

Therefore, by Lemma~\ref{lem:makebip}, there exist at least $\alpha' |Y|^s \ge \alpha' (d|V_i|)^s = \alpha |V_i|^s$ sets $Z \subset Y \subset V_i$ of size $s$ that are $(u,v)$-completable, as claimed.
\end{claimproof}

We now show that if $G[X_i]$ has large chromatic number then there is a single set $Z\subset V_i$ which is $(u,v)$-completable for many edges $uv\in G[X_i]$. 

\begin{claim}\label{clm:third} For each $i \in [k]$, if $\chi\big( G[X_i] \big) > 2s/\alpha$ then there exists a set of edges $E \subset G[X_i]$ of average degree at least $2s$, and an $s$-set $Z \subseteq V_i$ which is $(u,v)$-completable for every $uv \in E$.
\end{claim}
\begin{claimproof}[Proof of Claim~\ref{clm:third}]
Since $G[X_i]$ is not $2s / \alpha$-colourable, it is not $\big(2s/\alpha - 1\big)$-degenerate, and so contains a subgraph of minimum degree, and hence also average degree, at least $2s / \alpha$. Let $E'$ be the edges of such a subgraph, and note that, by Claim~\ref{clm:second}, for each edge $uv \in E'$ there exist at least $\alpha |V_i|^s$ sets $Z \subset V_i$ of size $s$ that are $(u,v)$-completable. Therefore, by the pigeonhole principle, there exists a subset $E \subset E'$ and a set $Z \subset V_i$ of size $s$, such that $Z$ is $(u,v)$-completable for every $uv \in E$, and 
\[|E| \, \ge \, \bigg( \frac{\alpha |V_i|^s}{\binom{|V_i|}{s}} \bigg) \cdot |E'| \, \ge \, \alpha |E'|\,.\]
Since $E'$ has average degree at least $2s / \alpha$ it follows that $E$ has average degree at least $2s$, as claimed.
\end{claimproof}

We can now bound $\chi\big(G[X_i]\big)$ for each $i\in[k]$. 

\begin{claim}\label{clm:fourth} $\chi\big( G[X_i] \big) \le 2s / \alpha$ for each $1 \le i \le k$.
\end{claim}
\begin{claimproof}[Proof of Claim~\ref{clm:fourth}]
Suppose that $\chi\big( G[X_i] \big) > 2s / \alpha$ for some $i \in [k]$, and let $E \subset G[X_i]$ and $Z \subseteq V_i$ be given by Claim~\ref{clm:third}. Thus $E$ has average degree at least $2s$, and for each $uv\in E$ the set $Z$ is $(u,v)$-completable. Letting $W$ be the set of vertices in edges of $E$, it follows that $Z$ has at least $s$ common neighbours in $N(u)$ for every vertex $u \in W$. We will show that $H \subset G$, which will contradict our assumption that $G$ is $H$-free, and hence prove the claim. 

Recall first that since $H$ is a cloud-forest graph (using the alternative definition), its vertex set can be partitioned into independent sets $I$ and $J$, and a forest $F'$, such that there are no edges from $V(F')$ to $I$ and each vertex of $J$ has at most one neighbour in $V(F')$.

By Fact~\ref{prop:forest}, $E$ contains $F'$. We now construct an embedding of $H$ into $G$ as follows. We embed the copy of $F'$ in $H$ into that in $E$. We then embed $I$ into vertices of $Z$ outside the copy of $F'$, and finally embed the vertices of $J$ greedily, preserving the property of having a graph embedding. We can embed $I$ to $Z$ because $Z$ has $s=v(H)$ vertices and there are no edges of $H$ from $V(F')$ to $I$. Finally, each vertex of $J$ has at most one neighbour in $V(F')$ and the rest of its neighbours are in $I$, so, since $Z$ has at least $v(H)$ common neighbours in $N(u)$ for each $u$ in the copy of $F'$, the greedy embedding of $J$ succeeds.
\end{claimproof}

It remains to prove that $X_0 = \emptyset$. We start by showing that if this is false, then there is a dense regular pair $(V_i,V_j)$ in $G$ and a substantial number of vertices with many neighbours in each of $V_i$ and $V_j$.

\begin{claim}\label{clm:fifth} 
If $X_0\neq\emptyset$, then there exists an $(\eps,d,p)$-regular pair $(V_i,V_j)$ and at least $2dpn / 3$ vertices with at least $dp|V_i| / 3$ neighbours in each of $V_i$ and $V_j$.
\end{claim}

\begin{claimproof}[Proof of Claim~\ref{clm:fifth}]
Suppose that $u \in X_0$, and recall that therefore
$$\big| N^*_2(u) \cap V_j \big| < \bigg( \frac{1}{2} + d \bigg) |V_j|$$
for every $j \in [k]$. Since $\delta(G) \ge pn / 3$, we can fix a set $U \subset N(u)$ of size $pn/3$, and for each $j \in [k]$, choose a subset $V'_j \subset V_j \setminus U$ of size $\big(\tfrac12+d\big)|V_j|$ containing $N^*_2(u) \cap V_j \setminus U$. 

We will first show (via some simple counting) that there exists $ij\in E(R)$ such that 
\begin{equation}\label{eq:eGViVj}
e\big( G[U,V'_i] \big) + e\big( G[U,V'_j] \big) \, > \, (1+5d) p |U| |V_i'|.
\end{equation}
Indeed, let $i$ be such that $e\big( G[U,V'_i] \big)$ is maximised, and consider $j\in V(R)$, not necessarily adjacent to $i$. Note that if there exist $2k/3$ indices $j$ such that the inequality~\eqref{eq:eGViVj} holds, then by~\eqref{eq:deltaR}, which bounds the minimum degree of $R$, at least one among them satisfies $ij\in E(R)$ and we are done. So let us suppose (for a contradiction) that there exist $k/3$ indices $j$ such that~\eqref{eq:eGViVj} fails to hold. Then, by the maximality of $e\big( G[U,V'_i] \big)$, we have
\begin{equation}\label{eq:eGUV:lower}
\sum_{j=1}^k e\big( G[U,V'_j] \big) \, \le \, \frac{k}{3} \cdot \Big( \big( 1 + 5d \big) p |U| |V_i'| - e\big( G[U,V'_i] \big) \Big) + \frac{2k}{3} \cdot e\big( G[U,V'_i] \big),
\end{equation}

We now establish a lower bound on the same sum. We have $\delta(G) \ge \big(\tfrac13+2\gamma\big)pn$, and 
\[e\big( G[U,V(G) \setminus N^*_2(u)] \big) \, \le \, dp^2n^2 \, = \, \frac{\gamma pn |U|}{2}\]
by the definition of $N^*_2(u)$ and since $|U| = pn/3$ and $d = \gamma/6$. Moreover, 
\[e\big( G[U] \big) \le p|U|^2 \ll pn |U| \qquad \text{and} \qquad e\big( G[U,V_0] \big) \le 2p|U||V_0| \le 2\eps pn|U|\]
by our assumptions~\ref{DNF:edgeint} and~\ref{DNF:edgecross}, and since $p = o(1)$. Now, since $N^*_2(u) \cap V_j \setminus U \subset V'_j$ for every $j \in [k]$, it follows from the above inequalities that
\begin{equation}\label{eq:eGUV:upper}
\sum_{j=1}^k e\big( G[U,V'_j] \big) \, \ge \, e\big( G[U,N^*_2(u) \setminus U] \big) - e\big( G[U,V_0] \big) \, \ge \, \bigg( \frac13 + \gamma \bigg) p n |U|.
\end{equation}
Now, combining~\eqref{eq:eGUV:lower} and~\eqref{eq:eGUV:upper}, we have
$$\bigg( \frac13 + \gamma \bigg) p n |U| \, \le \, \frac{k}{3} \cdot \Big( \big( 1 + 5d \big) p |U| |V_i'|  + e\big( G[U,V'_i] \big) \Big) \, \le \, \bigg( \frac{1}{3} + 2d \bigg) pn |U| ,$$
where the second inequality follows as $e\big( G[U,V'_i] \big) \, \le \, \big( 1 + d \big) p |U| |V'_i|$, 
by assumption~\ref{DNF:edgecross}, and since $|V_i'| = \big(\tfrac12+d\big)|V_i| \le \big(\tfrac12+d\big)n/k$. But recalling that $d = \gamma/6$, we see that this is a contradiction, and thus we have proved that there is a pair $ij \in E(R)$ for which~\eqref{eq:eGViVj} holds. 

Let us fix such a pair, and set
\begin{equation*}
U' \, = \, \bigg\{ w \in U : \min\big\{ |N(w) \cap V_i'|, |N(w) \cap V_j'| \big\} \ge \frac{dp|V_i|}{3} \bigg\}.
\end{equation*}
In order to prove the claim, it will suffice to show that $|U'| \ge 2d|U| = 2dpn/3$. Set
$$c \, := \, \frac{e\big( G[U,V'_i] \big)}{p \cdot |U| |V_i'|},$$
and observe that $c \le 1 + d$ by assumption~\ref{DNF:edgecross}. We claim that 
\begin{equation}\label{eq:Vineighbours}
\big| \big\{ w \in U : \ |N(w) \cap V_i'| \ge dp|V_i'| \big\} \big| \, \ge \, (c - 2d)|U|
\end{equation}
and 
\begin{equation}\label{eq:Vjneighbours}
\big| \big\{ w \in U : \ |N(w) \cap V_j'| \ge dp|V_i'| \big\} \big| \, \ge \, (1 - c + 4d)|U|,
\end{equation}
from which it will follow (by inclusion-exclusion, and since $|V_i'| > |V_i| / 3$) that we have $|U'| \ge 2d|U|$, as required. Suppose first that~\eqref{eq:Vineighbours} fails to hold, and observe that therefore
$$e\big( G[U,V'_i] \big) \, \le \, (c - 2d)|U| \cdot (1+d)p|V'_i| + (1 - c + 2d)|U| \cdot dp|V'_i| \, < \, cp|U||V'_i|$$
by assumption~\ref{DNF:edgecross}, which contradicts the definition of $c$. Similarly, if~\eqref{eq:Vjneighbours} fails to hold, then
$$e\big( G[U,V'_j] \big) \, \le \, (1 - c + 4d)|U| \cdot (1+d)p|V'_i| + (c - 4d)|U| \cdot dp|V'_i| \, = \, (1 - c + 5d)p|U||V'_i|,$$
again using assumption~\ref{DNF:edgecross}, which contradicts~\eqref{eq:eGViVj}. Hence, both~\eqref{eq:Vineighbours} and~\eqref{eq:Vjneighbours} hold, and so the claim follows. 
\end{claimproof}

We now use Claim~\ref{clm:fifth} to show that if $X_0\neq\emptyset$ then there exist a substantial number of vertices each of whose neighbourhoods is dense.

\begin{claim}\label{clm:sixth} 
If $X_0\neq\emptyset$, then there exists an $(\eps,d,p)$-regular pair $(V_i,V_j)$ and a set $W$ of size $dpn / 3$ with the following property. For every $w \in W$, there exists a graph $G_w \subset G\big[ N(w) \cap V_i, N(w) \cap V_j \big]$ with $dp|V_i|/2$ vertices and at least $2^{-6}d^3p^3|V_i|^2$ edges.
\end{claim}

\begin{claimproof}[Proof of Claim~\ref{clm:sixth}]
By Claim~\ref{clm:fifth} there is an $(\eps,d,p)$-regular pair $(V_i,V_j)$ and a set $U'$ of $2dpn/3$ vertices, each of which has at least $dp|V_i|/3$ neighbours in each of $V_i$ and $V_j$. We first prove that there exists a subset $W \subset U'$ of size $dpn / 3$ such that the following hold for every $w \in W$:
\begin{enumerate}[label=\abc]
\item\label{sixth:b} $|N_\Gamma(w) \cap V_i| \le 2p|V_i|$ and $|N_\Gamma(w) \cap V_j| \le 2p|V_j|$,\smallskip
\item\label{sixth:c} the pair $\big(N_\Gamma(w) \cap V_i,N_\Gamma(w) \cap V_j \big)$ is $(\eps',d,p)$-lower-regular,
\end{enumerate}
where $\eps' = \eps'(d) > 0$ was chosen earlier to be sufficiently small. This will be sufficient to prove the claim, because it follows, by the Slicing Lemma applied with $\alpha = d/6$, that $\big(N(w) \cap V_i, N(w) \cap V_j\big)$ is $(6\eps'/d,d/2,p)$-lower-regular, and hence
$$e\big( G[X,Y] \big) \, \ge \, \frac{d^3 p^3 |V_i|^2}{2^6}$$
for every $X \subset N(w) \cap V_i$ and $Y \subset N(w) \cap V_j$ with $|X| = |Y| = dp|V_i|/4$. 

We will obtain the set $W$ by removing from $U'$ those vertices which fail either condition~\ref{sixth:b} or~\ref{sixth:c}, and then taking an arbitrary subset of the correct size. We claim that there are only $n^{o(1)}$ such `bad' vertices in $U'$. To see that there are only $n^{o(1)}$ vertices $w \in U'$ such that $|N_\Gamma(w) \cap V_i| > 2p|V_i|$ or $|N_\Gamma(w) \cap V_j| > 2p|V_j|$, simply observe that assumption~\ref{DNF:edgeint} implies that the number of such vertices is at most $\frac{2\log n}{p^2} = n^{o(1)}$. To prove the corresponding statement for condition~\ref{sixth:c}, we apply Theorem~\ref{thm:2reg} with $X = V_i$ and $Y = V_j$. Recall that the pair $(V_i,V_j)$ is $(\eps,d,p)$-lower-regular, and therefore, by Theorem~\ref{thm:2reg}, there are at most $C' \max\big\{ \tfrac{\log n}{p},p^{-2} \big\} = n^{o(1)}$ vertices $w \in V(G)$ such that $\big(N_\Gamma(w) \cap V_i, N_\Gamma(w) \cap V_j\big)$ is not $(\eps',d,p)$-lower-regular, as required. This completes the proof of the claim.
\end{claimproof}

We are finally ready to show that $X_0=\emptyset$.  

\begin{claim}\label{clm:seventh} 
$X_0 = \emptyset$.
\end{claim}

\begin{claimproof}[Proof of Claim~\ref{clm:seventh}]
Suppose for a contradiction that $X_0\neq\emptyset$, and let $(V_i,V_j)$ and $W$ be given by Claim~\ref{clm:sixth}. We will use Theorem~\ref{lem:densbip} to find many copies of $K_{s,s}$ in $G_w$ for each $w \in W$, which gives a lower bound on the number of copies of $K_{1,s,s}$ with parts in (respectively) $W$, $V_i$ and $V_j$. This bound will be sufficiently large that, via an application of the pigeonhole principle, we can find a copy of $K_{s,s,s}$ in $G$. But $H \subset K_{s,s,s}$, so this gives the desired contradiction.

To spell out the details, let $\beta > 0$ be the constant returned by Theorem~\ref{lem:densbip} with input $s$, and recall that, for every $w \in W$, the graph $G_w \subset G\big[ N(w) \cap V_i, N(w) \cap V_j\big]$ has $dp|V_i|/2$ vertices and $d^3 p^3 |V_i|^2 / 2^6$ edges. By Theorem~\ref{lem:densbip}, it follows that $G_w$ contains at least
$$\beta \bigg( \frac{dp}{16} \bigg)^{s^2} \bigg( \frac{dp|V_i|}{2} \bigg)^{2s}$$
copies of $K_{s,s}$, and hence, recalling that $|W| = dpn/3$, it follows that the number of pairs $(w,K)$, where $w \in W$ and $K \subset G\big[ N(w) \cap V_i, N(w) \cap V_j\big]$ is a copy of $K_{s,s}$, is at least
\begin{equation}\label{eq:countingpairs:wK}
\frac{dpn}{3} \cdot \beta\bigg( \frac{dp}{16} \bigg)^{s^2} \bigg( \frac{dp|V_i|}{2} \bigg)^{2s} \ge \, \beta \bigg( \frac{dp}{16} \bigg)^{(s+1)^2} |V_i|^{2s} n \, \gg \, s |V_i|^{2s},
\end{equation}
since $p = n^{-o(1)}$. 

Finally, note that there are at most $|V_i|^{2s}$ choices for the graph $K$, and so there must exist some $K$ which is in the neighbourhood of at least $s$ distinct $u \in W$. In particular, we have a copy of $K_{s,s,s}$ in $G$, and hence a copy of $H$. This contradiction proves the claim.
\end{claimproof}

By Claims~\ref{clm:fourth} and~\ref{clm:seventh}, it follows that $\chi(G) \le 2sk / \alpha$, which completes the proof of the proposition.
\end{proof}

\section{Upper bounds for \texorpdfstring{$p$}{p} constant}\label{sec:pconst}

In this section we will prove the upper bound $\delta_\chi(H,p) \le \delta_\chi(H)$ for all $p > 0$ constant. Together with the matching lower bound from Proposition~\ref{prop:pconstant:construction}, this completes the proof of Theorem~\ref{thm:pconst}. Recall that the bound $\delta_\chi(H,p) \le 1 - \frac{1}{\chi(H) - 1}$ for all $H$ and any constant $p$ was proved in Corollary~\ref{cor:turan}, so that it remains to prove the following two propositions.

Recall that the decomposition family of a graph $H$ is the collection of bipartite graphs obtained from $H$ by removing $\chi(H) - 2$ independent sets. The following proposition generalises~\cite[Theorem~7]{ChromThresh}, which proved it for $p=1$.

\begin{prop}\label{prop:pconstant:middle}
Let $H$ be a graph with $\chi(H) = r \ge 3$, and let $0 < p \le 1$ be a constant. If $H$ has a forest in its decomposition family, then
\[
\delta_\chi(H,p) \, \le \, \delta_\chi(H)\,=\,\frac{2r-5}{2r-3}.
\]
That is, for every $\gamma > 0$, there exists $C = C(H,p,\gamma) > 0$ such that, with high probability, every $H$-free spanning subgraph $G \subset G(n,p)$ with $\delta(G) \ge \big( \frac{2r-5}{2r-3} + \gamma \big) pn$ satisfies $\chi(G) \le C$. 
\end{prop}

Recall that a graph $H$ is near-acyclic if $\chi(H)=3$ and~$H$ admits a partition into a forest~$F$ and an independent set~$I$ such that every odd cycle of~$H$ meets~$I$ in at least two vertices. It is $r$-near-acyclic if it is possible to obtain a near-acyclic graph by removing $\chi(H)-3$ independent sets from $H$. The following proposition generalises~\cite[Theorem~34]{ChromThresh}.

\begin{prop}\label{prop:pconstant:bottom}
 Let $H$ be an $r$-near-acyclic graph with $\chi(H) = r \ge 3$, and let $0 < p \le 1$ be a constant. Then we have
 \[\delta_\chi(H,p) \, \le \, \delta_\chi(H)\,=\,\frac{r-3}{r-2}.\]
 That is, for every $\gamma > 0$, there exists $C = C(H,p,\gamma) > 0$ such that, with high probability, every $H$-free spanning subgraph $G \subset G(n,p)$ with $\delta(G) \ge \big( \frac{r-3}{r-2} + \gamma \big) pn$ satisfies $\chi(G) \le C$. 
\end{prop}

We emphasise that, in both propositions, the bound on the chromatic number is allowed to depend on $p$. Indeed, the construction used to prove~\cite[Theorem~1.3]{sparse} shows that, for every $\gamma > 0$ and $C$, if $\chi(H) = r \ge 4$ and $p = p(H,\gamma) > 0$ is sufficiently small, then with high probability $G(n,p)$ contains spanning $H$-free subgraphs with minimum degree at least $\big( \frac{r-2}{r - 1} - \gamma \big) pn$ and chromatic number at least $C$.

The proofs of both propositions are, with the exception of one key step in both proofs, relatively straightforward modifications of the proofs in~\cite{ChromThresh} of the corresponding bounds in the case $p = 1$. For this reason we will emphasise the new ideas required in the random setting, and refer the reader to~\cite{ChromThresh} for motivation of the lemmas we quote from~\cite{ChromThresh}. We will first (in Section~\ref{sec:newlemma}) state and prove a lemma which is the main new tool we will require; then (in Sections~\ref{sec:middle} and~\ref{sec:bottom}) we will describe how we combine this lemma with the method of~\cite{ChromThresh} in order to prove the propositions.

\subsection{A bound on the number of pairs with few common neighbours}\label{sec:newlemma}

Let $G$ be a subgraph of $G(n,p)$, and suppose that $U$ and $W$ are sets of vertices with the property that every $u \in U$ has more than $\big( \tfrac12 + \gamma \big) p|W|$ neighbours in $W$. If $p=1$, then this implies that every pair of vertices has many common neighbours in $W$, but for $p < 1$ there may exist some exceptional pairs with small common neighbourhood, even if $W$ is quite large. 

The following lemma, which is the key new step in the proof of both of the propositions above, says that (with high probability) few pairs in $U$ have small common neighbourhood in $W$ for every such pair $(U,W)$ with $|W| = \Omega(n)$. As in the previous section, we will write $N_\Gamma(u)$ for the neighbourhood of $u$ in $\Gamma = G(n,p)$ and $N(u)$ for the neighbourhood in $G$. 

\begin{lemma}\label{lem:mostgood} 
Given $p,\gamma,\alpha \in (0,1]$ there exists $C > 0$ such that, with high probability, the following holds. For every subgraph $G \subset G(n,p)$, and every pair of vertex sets $U$ and $W$ satisfying $|W| \ge \alpha n$ and 
\begin{equation}\label{eq:mostgood:condition}
|N(u)\cap W| \ge \bigg( \frac{1}{2} + \gamma \bigg) p |W|
\end{equation}
for every $u \in U$, we have
$$|N(u) \cap N (v) \cap W| \le \gamma p^2 |W|$$ 
for at most $C|U|$ pairs $u,v \in U$.
\end{lemma}

The first step to proving this lemma is to find a bounded-size subset $X \subset U$ such that the $\Gamma$-neighbourhoods in $W$ of subsets of $U \setminus X$ are all well-behaved. The following lemma is quite a bit more general (though no harder to prove) than the result we need. Note that for this lemma we require $p<1$.

\begin{lemma} \label{boundedweirdness} 
Given $p\in(0,1)$ and $\delta \in (0,1]$, there exists $C > 0$ such that, with high probability, the following holds. For every set of vertices $W \subseteq V = V\big(G(n,p)\big)$, there exists a set of vertices $X_W \subseteq V$ with $|X_W| \le C$ such that 
\begin{equation}\label{eq:weirdness}
\Big| \bigcap_{v \in S} N_\Gamma(v)\, \cap\, W \Big| \, \in \, p^{|S|} |W| \, \pm\, \delta n
\end{equation}
for all $S \subseteq V \setminus X_W$. 
\end{lemma}

\begin{proof}
Define $\ell = \ell(p,\delta)$ to be the least positive integer such that $p^{\ell} < \delta/2$. Given $W$, let $\mathcal{X}$ be a maximal family of disjoint vertex sets satisfying $|S| \le \ell$ and
\begin{equation}\label{eq:weird:bad}
\Big| \bigcap_{v\in S} N_\Gamma(v)\, \cap \, W \Big| \not\in p^{|S|}|W| \pm \frac{\delta n}{2}
\end{equation}
for each $S \in \X$. We claim that $X_W=\bigcup_{S\in \mathcal{X}}S$ satisfies the required conditions. Indeed, for any subset $S\subseteq V \setminus X_W$ with $|S|\le \ell$ it is immediate from the maximality of $\mathcal{X}$ that~\eqref{eq:weirdness} holds. On the other hand, if $|S|>\ell$ then let $S_0$ be an arbitrary subset of $S$ of cardinality $\ell$, and observe that
\[
0\, \le \, \Big|\bigcap_{v\in S}N_\Gamma(v)\, \cap \, W \Big| \, \le\, \Big|\bigcap_{v\in S_0}N_\Gamma(v)\, \cap \, W \Big|\, \le \, p^{\ell}|W| +\frac{\delta n}{2}\, \le \, \delta n,
\]
as required.

To show that, with high probability, $|X_W| \le C$ for every set $W$, note first that, by the Chernoff bound, 
\[
\Prob\bigg( \Big| \bigcap_{v\in S} N_\Gamma(v)\, \cap \, W \Big| \not\in p^{|S|}|W| \pm \frac{\delta n}{2} \bigg)\, \le\, e^{-\Omega(\delta n)}.
\]
These events are moreover independent for disjoint sets $S$, so if $\X$ contains $t$ sets for some sufficiently large constant $t$, the probability that all of them satisfy~\eqref{eq:weird:bad} is at most $e^{-n}$. Any $\X$ with $t$ sets gives $X_W$ with at most $C=t\ell$ vertices. The number of ways to choose such an $\X$ is at most $n^{O(1)}$, and the number of ways to choose $W$ is at most $2^n$, so by the union bound we conclude that with high probability, for all $W$ we have $|X_W|\le C$ as desired.
\end{proof}

We can now prove Lemma~\ref{lem:mostgood}.

\begin{proof}[Proof of Lemma~\ref{lem:mostgood}]
For $p=1$, if $U$ and $W$ satisfy the conditions of the lemma, then for any $u,v\in U$ we have $|N(u)\cap N(v)\cap W|\ge 2\gamma|W|-2>\gamma p^2|W|$, so the lemma statement is true. We thus assume from now on that $p<1$.

Fix $\delta = \delta(\alpha,\gamma,p) > 0$ sufficiently small, choose $C$ sufficiently large for Lemma~\ref{boundedweirdness} to hold, and suppose that $\Gamma=G(n,p)$ has the property described in that lemma. Let us assume also that there do not exists sets $S,T\subset V(\Gamma)$ such that 
$$|S| \ge \frac{\alpha\gamma pn}{5}, \qquad |T|\ge C  \qquad \text{and} \qquad  e\big( \Gamma[S,T] \big) \le \left( 1 - \frac{\gamma}{10} \right) p|S||T|,$$
and note that this also holds with high probability, by Chernoff's inequality. 

Now let $G\subset\Gamma$ and let $U$ and $W$ satisfy the conditions of the lemma. Let $X_W$ be the set given by Lemma~\ref{boundedweirdness}. Recall that for each $u\in U$ we have $\big|N(u)\cap W\big|\ge\big(\tfrac12+\gamma\big)p|W|$. For each $u\in U$ we define a set $X_u$ of all vertices~$v \in U \setminus X_W$ (so $N_\Gamma(u)\cap N_\Gamma(v)\cap W$ has about the expected size) such that $N(u)\cap N_\Gamma(v)\cap W$ is a bit smaller than expected:
\begin{equation}\label{def:Xu}
X_u \, := \, \left\{ v \in U \setminus X_W : | N(u) \cap N_\Gamma(v) \cap W|\le \bigg( \frac{1}{2} + \frac{3\gamma}{5} \bigg) p^2 |W| \right\}\,.
\end{equation}

\setcounter{claim}{0}

\begin{claim}\label{clm:mostgoodone} If $u \in U \setminus (X_W \cup X_v)$ and $v \in U \setminus (X_W \cup X_u)$, then
$$|N(u)\cap N(v) \cap W| \, \ge \, \gamma p^2 |W|.$$
\end{claim}
\begin{claimproof}[Proof of Claim~\ref{clm:mostgoodone}]
Since $u,v \not\in X_W$, and assuming we chose $\delta < \alpha \gamma p^2 / 5$, we have
\[
|N_\Gamma(u)\cap N_\Gamma(v)\cap W| 
\, \le \, p^2 |W|+\delta n
\, \le \, \bigg( 1 + \frac{\gamma}{5} \bigg) p^2 |W|.
\]
Also, since $u\not\in X_v$ and $v\not\in X_u$, it follows that 
$$\min\Big\{ | N(u)\cap N_\Gamma(v)\cap W |, \, | N(v)\cap N_\Gamma(u)\cap W| \Big\} \, \ge \, \bigg( \frac{1}{2} + \frac{3\gamma}{5} \bigg) p^2 |W|.$$ 
Moreover, both are subsets of $N_\Gamma(u)\cap N_\Gamma(v)\cap W$, and so, by inclusion-exclusion, 
\[
|N(u)\cap N(v) \cap W| \,\ge \, 2\bigg( \frac{1}{2} + \frac{3\gamma}{5} \bigg) p^2 |W| \, - \, \bigg( 1 + \frac{\gamma}{5} \bigg) p^2 |W| \, = \, \gamma p^2 |W|,
\]
as claimed.
\end{claimproof}

\begin{claim}\label{clm:mostgoodtwo}  $|X_u| \le C$ for every $u \in U \setminus X_W$. 
\end{claim}

\begin{claimproof}[Proof of Claim~\ref{clm:mostgoodtwo}]
We will count the paths of length two from $u$ to $X_u$ with first edge in $G$, second edge in $G(n,p)$, and middle vertex in $W$. By~\eqref{def:Xu}, the number of such paths is at most
\begin{equation}\label{eq:paths:upper}
\bigg( \frac{1}{2} + \frac{3\gamma}{5} \bigg) p^2 |W| \cdot |X_u|.
\end{equation}
On the other hand, defining 
\[
\hat{W}_u := \left\{ w \in W \,:\, | N_\Gamma(w) \cap X_u | \le \bigg( 1 - \frac{\gamma}{10} \bigg) p |X_u|  \right\},
\]
we obtain a lower bound on the number of such paths of
\begin{equation}\label{eq:paths:lower}
\big| N(u) \cap \big( W \setminus \hat{W}_u \big) \big| \cdot \bigg( 1 - \frac{\gamma}{10} \bigg) p |X_u| \, 
\geByRef{eq:mostgood:condition} \, \bigg( \bigg( \frac{1}{2} + \gamma \bigg) p |W| - |\hat{W}_u| \bigg) \bigg( 1 - \frac{\gamma}{10} \bigg) p |X_u|.
\end{equation}
Now if $|\hat{W}_u| \le \gamma p |W| / 5$, then using~\eqref{eq:paths:upper} and~\eqref{eq:paths:lower}, we obtain 
\[\bigg( \frac{1}{2} + \frac{4\gamma}{5} \bigg) p |W| \cdot \bigg( 1 - \frac{\gamma}{10} \bigg) p |X_u| \le \bigg( \frac{1}{2} + \frac{3\gamma}{5} \bigg) p^2 |W| \cdot |X_u|\,,\] 
which is a contradiction. We conclude that $|\hat{W}_u|>\gamma p|W|/5$, so by definition of $\hat{W}_u$ and by our assumption on $\Gamma$ (with $S=\hat{W}_u$ and $T=X_u$), we have $|X_u|\le C$, as claimed.
\end{claimproof}

The lemma follows easily from the claims. Indeed, by Claim~\ref{clm:mostgoodone}  the only pairs $u,v \in U$ with $|N(u)\cap N (v) \cap W| \le \gamma p^2 |W|$ are those with $u \in X_W \cup X_v$ or $v \in X_W \cup X_u$. Since $|X_W| + |X_u| = O(1)$ for every $u \in U$, by Lemma~\ref{boundedweirdness} and Claim~\ref{clm:mostgoodtwo} , it follows that the number of such pairs is at most $O(|U|)$, as required.
\end{proof}

\subsection{The proof of Proposition~\ref{prop:pconstant:middle}}\label{sec:middle}

We need the following two lemmas from~\cite{ChromThresh}. Let $K_\ell(t)$
be the $t$-blow-up of $K_\ell$, that is, the graph obtained from $K_\ell$
by replacing each vertex with an independent set of size~$t$, and each edge
with a complete bipartite graph $K_{t,t}$. We write $F+H$ for the join of $F$ and $H$, that is, the graph obtained from a disjoint union of $F$ and $H$ by adding all edges between $F$ and $H$.

\begin{lemma}[Lemma~9 of~\cite{ChromThresh}]\label{pigeon}
Let $\alpha > 0$ and $3 \le r,t \in \N$, let~$F$ be a forest, and let~$H \subset F + K_{r-2}(t)$. Let $G$ be a graph on~$n$ vertices, and $X \subset V(G)$. If $G[N(u)]$ contains at least $\alpha n^{(r-1)v(H)}$ copies of $K_{r-1}(v(H))$ for every $u \in X$, then either $H \subset G$ or $|X| \le v(H)/\alpha$.
\end{lemma}

\begin{lemma}[Lemma~10 of~\cite{ChromThresh}]\label{lem:extend}
For every $\alpha > 0$ and $3 \le r,t \in \N$, there exists $C = C(\alpha,r,t)$ such that for every graph $H \subset F + K_{r-2}(t)$, the following holds. Let $G$ be an $H$-free graph on~$n$ vertices, and let $X \subset V(G)$ be such that every edge of $G[X]$ is contained in at least $\alpha n^{r-2}$ copies of $K_r$ in~$G$. Then~$G[X]$ is $Cv(F)$-degenerate.
\end{lemma}

We can now give the proof.

\begin{proof}[Proof of Proposition~\ref{prop:pconstant:middle}]
Let $F$ be a forest in the decomposition family of $H$, so in particular $H \subset F + K_{r-2}\big(v(H)\big)$. Let $\gamma > 0$, and let $G$ be an $H$-free spanning subgraph of $G(n,p)$ with
$$\delta(G) \ge \bigg( \frac{2r - 5}{2r-3} + 3\gamma \bigg) pn.$$
Applying the sparse minimum degree form of Szemer\'edi's Regularity Lemma, with $k_0 = r / \gamma$, $d = \gamma$ and $\eps$ sufficiently small, to $G$, we obtain a partition $V(G) = V_0 \cup V_1 \cup \cdots \cup V_k$ with $k_0\le k\le k_1=O(1)$ such that the $(\eps,d,p)$-reduced graph $R$ satisfies
$$\delta(R) \ge \bigg( \frac{2r - 5}{2r-3} + \gamma \bigg) k\,.$$
We now partition the vertices of $V(G)$ according to the sets $V_i$ to which they send `many' edges. More precisely, define for each $I_2\subseteq I_1\subseteq [k]$, 
\begin{equation}\label{def:XI1I2}\begin{split}
X_{I_1,I_2} \, := \, \Big\{ v \in V(G) \,:\, I_1 & = \big\{ i \in [k] : |N(v) \cap V_i| \ge \gamma p |V_i| \big\}\,,  \\ 
I_2 & = \big\{ i \in [k] : |N(v)\cap V_i| \ge \big( \tfrac12 + \gamma \big)p |V_i| \big\} \Big\}\,.
\end{split}\end{equation}
We remark that this is a refinement of the partition used in the proof of~\cite[Theorem~7]{ChromThresh}.
Since the number of parts in this partition is at most $3^{k_1}$, the following claim completes the proof of Proposition~\ref{prop:pconstant:middle}.

\begin{claim*} $\chi\big(G[X_{I_1,I_2}]\big) =O(1)$ for every $I_2 \subseteq I_1 \subset [k]$.
\end{claim*}

\begin{claimproof}[Proof of Claim]
We partition the edges of $G[X_{I_1,I_2}]$ into graphs $J = J_{I_1,I_2}$ and $L = L_{I_1,I_2}$, where $J$ consists of those edges whose endpoints have at least $\gamma p^2 |V_i|$ common neighbours in each of the clusters $\{ V_i : i \in I_2 \}$ and the remaining edges are in $L$. We will show that $J$ and $L$ are both $O(1)$-degenerate, which implies both are $O(1)$-colourable and hence that $G[X_{I_1,I_2}]$ is $O(1)$-colourable, as desired.

We start with $L$. For each $U \subset X_{I_1,I_2}$ and $i \in I_2$,
consider the graph $L_i[U] \subset L[U]$ on vertex set $U$ formed by the
edges of $G$ that have fewer than $\gamma p^2 |V_i|$ common neighbours in
$V_i$. Setting $W = V_i$, note that the pair $(U,W)$ satisfies~\eqref{eq:mostgood:condition}, by the definition of $X_{I_1,I_2}$ and since $i \in I_2$. By Lemma~\ref{lem:mostgood}, it follows that $L_i[U]$ has at most $C|U|$ edges (for some constant $C = C(H,p,\gamma)$) and hence has bounded average degree. Since $L[U] \subset \bigcup_{i \in I_2} L_i[U]$, and $|I_2| \le k$, it follows that the graph $L[U]$ has bounded average degree. Since this holds for every $U \subset X_{I_1,I_2}$, it follows that $L$ is $O(1)$-degenerate, as claimed.

The proof that $J$ is $O(1)$-degenerate is almost the same as
in~\cite[Theorem~7]{ChromThresh}, and so we shall only sketch the proof, skipping most of the details. Suppose first that 
$$|I_1| \, \ge \, \left( \ds\frac{2r-4}{2r-3} \right) k\,.$$
In this case we shall show that $|X_{I_1,I_2}| = O(1)$, and thus $J$ is trivially $O(1)$-degenerate. We first claim that $R[I_1]$ contains a copy of $K_{r-1}$. Indeed, by our minimum degree condition on $R$, we have $$\delta\big( R[I_1] \big) \, \ge \, \delta(R) - \big( k - |I_1| \big) \, \ge \, |I_1| - \left( \frac{2}{2r-3} - \gamma \right) k \, \ge \, \left( \frac{r-3}{r-2} + \gamma \right) |I_1|\,,$$
so $R[I_1]$ contains a copy of~$K_{r-1}$, as claimed. Let $\{W_1,\dots,W_{r-1}\} \subset \{V_i : i \in I_1\}$ be the set of parts corresponding to this copy of~$K_{r-1}$.

Now let~$u \in X_{I_1,I_2}$, and recall that $|N(u) \cap W_i|\ge \gamma p
|V_i|$ for each $i \in [r-1]$, by the definition of~$X_{I_1,I_2}$. By the
(dense) Slicing and Counting Lemmas, it follows that $G[N(u)]$ contains at least $\Omega\big( n^{(r-1)v(H)} \big)$ copies of $K_{r-1}(v(H))$. Since $H\not\subset G$, by Lemma~\ref{pigeon} we have $|X_{I_1,I_2}| = O(1)$, as claimed.
  
So let us assume from now on that
\[|I_1| \, < \, \left( \ds\frac{2r-4}{2r-3} \right) k\,,\]
and that $G(n,p)$ has the following property: for each vertex set $S$ of
size at least $n/(2k_1)$, the number of vertices that have more than
$(1+\gamma)p|S|$ neighbours in $S$ is at most a constant depending on $p$,
$\gamma$ and $k_1$ (for constant~$p$ this holds with high probability by the
Chernoff bound). In particular this applies for $S=V_i$ for each $1\le i\le k$. Now if there is no vertex of $X_{I_1,I_2}$ which has at most $(1+\gamma)p|V_i|$ neighbours in each $V_i$, then this implies $|X_{I_1,I_2}|=O(1)$, in which case $J$ is trivially $O(1)$-degenerate. So suppose $u \in X_{I_1,I_2}$ is a vertex with at most $(1+\gamma)p|V_i|$ neighbours in $V_i$ for every $0 \le i \le k$. 

We claim that $R[I_2]$ contains a copy of~$K_{r-2}$. Indeed, since $\delta(G) \ge \big(\tfrac{2r-5}{2r-3}+3\gamma\big)pn$, and since $|V_i| \le n/k$ for each $i \in [k]$ and $|V_0| \le \eps n \le \gamma n / 2$, it follows that
\[\bigg( \frac{2r-5}{2r-3} + 3\gamma \bigg)pn \, \le \, d(u) \, \le \, \bigg( (1+\gamma) |I_2| + \left( \frac12+\gamma \right)\big( |I_1| - |I_2| \big) + 2\gamma k \bigg) \frac{pn}{k}.\]
Hence, using our bound on $|I_1|$, it follows that $|I_2| \ge \big( \frac{2r-6}{2r-3} \big) k$. Thus
\[\delta\big( R[I_2] \big) \, \ge \, \delta(R) - \big( k - |I_2| \big) \, \ge \, |I_2| - \left( \frac{2}{2r-3} - \gamma \right) k \, \ge \, \left( \frac{r-4}{r-3} + \gamma \right) |I_2|\,,\]
so $R[I_2]$ contains a copy of~$K_{r-2}$, as claimed. 

Finally, recall that the endpoints of each edge of $J$ have at least $\gamma p^2 |V_i|$
common neighbours in $V_i$ for every $i \in I_2$. By the (dense) Slicing
and Counting Lemmas, it follows that each edge of $J$ is contained in
$\Omega(n^{r-2})$ copies of $K_r$. Since $H\not\subset G$,
Lemma~\ref{lem:extend} tells us that $J$ is $O(1)$-degenerate. 
We have thus proved that both $J$ and $L$ are $O(1)$-degenerate, which implies that $\chi\big(G[X_{I_1,I_2}]\big) = O(1)$, as claimed.
\end{claimproof}

As noted earlier, the sets $X_{I_1,I_2}$ partition the vertex set of $G$ into at most $3^{k_1}$ parts, and hence the claim implies that $\chi(G) = O(1)$, as required. 
\end{proof}

\subsection{The proof of Proposition~\ref{prop:pconstant:bottom}}\label{sec:bottom}

We need the following lemma, which was essentially proved
in~\cite[Section~7]{ChromThresh}. The difference here is that we have to
explicitly assume~\ref{magic:d} rather than deducing it from a degree condition.

\begin{lemma}\label{lem:magic}
For each $r \ge 3$ and $\gamma > 0$, and each $r$-near-acyclic graph $H$ with $\chi(H) = r$, there exist $\eps > 0$, $C > 0$ and $m_0 \in \N$ such that the following holds for every $m \ge m_0$. If $G$ is a graph containing pairwise disjoint vertex sets $X, Y, Z_1,\dots,Z_{r-3}$ with the following properties:
\begin{enumerate}[label=\abc]
\item $\chi\big( G[X] \big) > C$,
\item $|Y| = |Z_1| = \dots = |Z_{r-3}| = m$,
\item $|N(v) \cap Y| \ge \gamma m$ for every $v \in X$, 
\item\label{magic:d} $|N(u) \cap N(v) \cap Z_i| \ge \gamma m$ for every $uv \in E(G[X])$ and every $i \in [r-3]$, 
\item each pair from $Y,Z_1,\ldots,Z_{r-3}$ forms an $\eps$-regular pair in $G$ of density at least $\gamma$,
\end{enumerate} 
then $H\subseteq G$.
\end{lemma}

The proof of Lemma~\ref{lem:magic} is roughly as follows. We proceed as in
the proof of~\cite[Theorem~34]{ChromThresh}: We
apply~\cite[Proposition~26]{ChromThresh} (the so-called `paired
VC-dimension' argument),~\cite[Lemma~24]{ChromThresh} (an inductive double
counting argument) and~\cite[Proposition~36]{ChromThresh} (which uses the
Counting Lemma and the pigeonhole principle), followed
by~\cite[Lemma~25]{ChromThresh} (which uses the fact that high degree
graphs contain all trees). The only point where some extra care is needed is in the application of~\cite[Proposition~36]{ChromThresh}, since this
proposition requires that every vertex of the set $X$ has at least $\big(\tfrac12+\gamma\big)m$
neighbours in each set $Z_i$. However, the only use made of this condition is to
deduce that each edge of $G[X]$ has a common neighbourhood of size at least $\gamma m$ in each $Z_i$, which is assumption~\ref{magic:d} above, so the conclusion we need follows from exactly the same proof. For a complete proof of Lemma~\ref{lem:magic}, see {\ifarxiv~Appendix~\ref{app:magic}\else~\cite[Appendix~A]{dense:arXiv}\fi}.

The deduction of Proposition~\ref{prop:pconstant:bottom} from Lemma~\ref{lem:magic} follows the same outline as the proof of Proposition~\ref{prop:pconstant:middle} above.

\begin{proof}[Proof of Proposition~\ref{prop:pconstant:bottom}]
Let $H$ be $r$-near-acyclic, let $\gamma > 0$, and let $G$ be an $H$-free spanning subgraph of $G(n,p)$ with
\begin{equation}\label{eq:mindegreeG}
\delta(G) \ge \bigg( \frac{r - 3}{r - 2} + 2\gamma \bigg) pn\,.
\end{equation}
Applying the sparse minimum degree form of Szemer\'edi's Regularity Lemma to~$G$, with $k_0 = r / \gamma$, $d = \gamma$ and $\eps$ sufficiently small, we obtain a partition $V(G) = V_0 \cup V_1 \cup \cdots \cup V_k$ with $k_0\le k\le k_1=O(1)$ such that the reduced graph $R$ satisfies
\[\delta(R) \ge \bigg( \frac{r - 3}{r - 2} + \gamma \bigg) k\,.\]
We define sets $X_{I_1,I_2}$ for each $I_2 \subseteq I_1\subseteq [k]$
exactly as in~\eqref{def:XI1I2}. 
Since, again, the number of parts $X_{I_1,I_2}$ is at most $3^{k_1}$, the
following claim completes the proof of the proposition. 

\begin{claim*} $\chi\big(G[X_{I_1,I_2}]\big) =O(1)$ for every $I_2 \subseteq I_1 \subset [k]$.
\end{claim*}
\begin{claimproof}[Proof of Claim]
We partition the edges of $G[X_{I_1,I_2}]$ into graphs $J = J_{I_1,I_2}$
and $L = L_{I_1,I_2}$, exactly as in the proof of
Proposition~\ref{prop:pconstant:middle}. That is, we let $J$ consist of
those edges whose endpoints have at least $\gamma p^2 |V_i|$ common
neighbours in each of the clusters $\{ V_i : i \in I_2 \}$. The proof that
$L$ is $O(1)$-degenerate (using Lemma~\ref{lem:mostgood}) is exactly the
same as before, since it does not use the minimum degree condition on~$G$
(and thus~$R$), and we omit it. 

The proof that $J$ is $O(1)$-chromatic is somewhat different to that in Section~\ref{sec:middle}, so let us give the details. As before, with high probability for each $S$ of size at least $n/(2k_1)$, the number of vertices in $G(n,p)$ with more than $(1+\gamma)p|S|$ neighbours in $S$ is $O(1)$ and in particular the number of vertices with more than $(1+\gamma)p|V_i|$ neighbours in any $V_i$ with $0\le i\le k$ is $O(1)$. Again, this implies that either $|X_{I_1,I_2}|=O(1)$ and thus $J$ is trivially $O(1)$-degenerate, or there exists $u\in X_{I_1,I_2}$ with at most $(1+\gamma)p|V_i|$ neighbours in $V_i$ for every $0 \le i \le k$. In this latter case we claim that $R[I_1]$ contains a pair of disjoint copies of $K_{r-2}$, each with at least $r-3$ vertices in $I_2$. 

Indeed, counting neighbours of $u$, we obtain
\[\bigg( (1+\gamma) |I_1| + \gamma \cdot \big( k - |I_1| \big) + 2\eps k \bigg) \frac{pn}{k} \, \ge \, d(u) \, \ge \, \bigg( \frac{r-3}{r-2} + 2\gamma \bigg)pn\,,\]
and hence $|I_1| \ge \big( \frac{r-3}{r-2} \big) k$. Similarly, 
\[\bigg( (1+\gamma) |I_2| + \left(\frac{1}{2} + \gamma \right) \big( k - |I_2| \big) + 2\eps k \bigg) \frac{pn}{k} \, \ge \, \bigg( \frac{r-3}{r-2} + 2\gamma \bigg)pn\,,\]
which implies that $|I_2| \ge \big( \frac{r-4}{r-2} \big) k$. By our minimum degree condition on $R$, it follows that we can choose (greedily) two disjoint cliques in $R[I_1]$ as claimed. 

Let the clusters corresponding to the vertices of our two cliques be
respectively $Y \in I_1$ and $Z_1,\dots,Z_{r-3} \in I_2$ (for one), and $Y' \in I_1$ and $Z'_1,\ldots,Z'_{r-3} \in I_2$ (for the other). Set 
$$X_1 = X_{I_1,I_2} \setminus \big( Y \cup Z_1 \cup \dots \cup Z_{r-3} \big),$$ 
and $X_2 = X_{I_1,I_2} \setminus X_1$, and suppose first that $\chi\big( J[X_1] \big) > C$, where $C = C(r,\gamma p^2)$ is the constant in Lemma~\ref{lem:magic}. Note that $X_1$ is disjoint from $Y, Z_1, \ldots, Z_{r-3}$, by definition, and moreover we have 
\begin{enumerate}[label=\abc]
\item $\chi\big( J[X_1] \big) > C$, by assumption,
\item $|Y| = |Z_1| = \dots = |Z_{r-3}| \ge n/2k$, since each is a part of the Szemer\'edi partition, 
\item $|N(v) \cap Y| \ge \gamma p |V_i|$ for every $v \in X_1$, by the definition of $X_{I_1,I_2}$,
\item $|N(u) \cap N(v) \cap Z_i| \ge \gamma p^2 |V_i|$ for every $uv \in E(J[X])$, by the definition of $J$, and
\item each pair from $Y,Z_1,\ldots,Z_{r-3}$ forms an $\eps$-regular pair in $G$ of density at least $\gamma p$, since $Z_1,\dots,Z_{r-3}$ and $Y$ form a clique in $R$, and $d = \gamma$.  
\end{enumerate} 
Therefore, by Lemma~\ref{lem:magic}, it follows that $H \subset G$, which is a contradiction. On the other hand, if $\chi\big( J[X_2] \big) > C$, then the same argument (with $Y, Z_1, \ldots, Z_{r-3}$ replaced by $Y', Z'_1, \ldots, Z'_{r-3}$) gives the same contradiction. Hence $\chi(J) \le 2C$, as required.
\end{claimproof}

As noted earlier, this completes the proof of Proposition~\ref{prop:pconstant:bottom}, and hence of Theorem~\ref{thm:pconst}. 
\end{proof}

\section{Determining \texorpdfstring{$\delta_\chi^*(H)$}{delta*(H)}}\label{sec:deltachistar}

In this section we sketch the proof of Theorem~\ref{thm:deltachistar}, which states that 
\begin{equation}\label{eq:deltachistar}
 \delta_\chi^*(H) = \min\Big\{\delta_\chi(H')\,:\, \text{there exists a homomorphism from }H\text{ to }H' \Big\}\end{equation}
for every graph $H$. We begin with the upper bound.

\begin{prop}\label{prop:deltachistar:upper}
If $H$ is any graph homomorphic to $H'$, then $\delta_\chi^*(H) \le \delta_\chi(H')$.
\end{prop}

We will use the following lemma, which follows from the result of Erd\H{o}s~\cite{Erdos64} that the Tur\'an density of any $k$-partite $k$-uniform hypergraph is zero. 

\begin{lemma}\label{lem:tblowup}
For any graph $H'$ and any $t \in \N$ and $c > 0$, if $n = v(G)$ is large enough and $G$ contains $c n^{v(H')}$ copies of $H'$, then $G$ contains the $t$-blow-up of $H'$.
\end{lemma}
\begin{proof}
 Take a uniform random partition of $V(G)$ into $v(H')$ parts, and let $F$ be a $v(H')$-uniform hypergraph on $V(G)$ whose edges correspond to copies of $H'$ in $G$ with the $i$th vertex of $H'$ in the $i$th part of the partition for each $i$. In expectation, $F$ contains at least $v(H')^{-v(H')}cn^{v(H')}$ edges, and by the result of Erd\H{o}s~\cite{Erdos64} any $F$ with so many edges contains a copy of the complete $v(H')$-partite hypergraph with parts of size $t$, giving the desired $t$-blow-up of $H'$ in $G$.
\end{proof}

\begin{proof}[Proof of Proposition~\ref{prop:deltachistar:upper}]
Note that $H$ is contained in the $v(H)$-blow-up of $H'$. Recall that, by the definition of $\delta_\chi(H')$, for each $\gamma > 0$ there exists $C = C(\gamma)$ such that any $H'$-free graph $G'$ with minimum degree at least $\big(\delta_\chi(H') + \gamma \big)v(G')$ has chromatic number at most $C$. We claim that, for any $\eps > 0$, every sufficiently large $H$-free graph $G$ with minimum degree at least $\big(\delta_\chi(H')+2\gamma\big)v(G)$ can be made $C$-partite by deleting at most $\eps v(G)^2$ edges. 

To see this, fix $\eps > 0$ and choose $\mu > 0$ sufficiently small. By the Graph Removal Lemma (see, e.g.,~\cite{KS93}), there exists $c > 0$ such that any $n$-vertex graph $G$ either contains at least $c n^{v(H')}$ copies of $H'$, or can be made $H'$-free by deleting at most $\mu n^2$ edges. By Lemma~\ref{lem:tblowup}, we conclude that for all sufficiently large $n$, the graph $G$ can be made $H'$-free by deleting at most $\mu n^2$ edges.

Let $G'$ be obtained from $G$ by deleting $\mu n^2$ edges in order to destroy all copies of $H'$, and then sequentially vertices of degree less than $\big(\delta_\chi(H')+\gamma\big)n$ until no more remain. Since $\mu$ was chosen sufficiently small, this process terminates having deleted fewer than $\eps n / 2$ vertices. Thus $G'$ is an $H'$-free graph with minimum degree at least $\big(\delta_\chi(H')+\gamma\big)v(G')$, and hence it has chromatic number at most $C$. Moreover, the total number of edges deleted from $G$ to obtain $G'$ is at most $\mu n^2 + \eps n^2 / 2 < \eps n^2$, as required.
\end{proof}

To complete the proof of~\eqref{eq:deltachistar}, we will show that a simple modification of the  constructions from~\cite{ChromThresh} suffices to prove the claimed lower bound on $\delta_\chi^*(H)$. We will use the following variant of Lemma~\ref{lem:constconst}, which also follows from Propositions~5 and~35 and Theorem~16 of~\cite{ChromThresh}.

\begin{lemma}\label{lem:construction:allsametime}
For every graph $H'$, integer $s \in \N$ and constants $C,\gamma > 0$, there exists $K = K(H',s,\gamma,C) > 0$ such that the following holds. For all sufficiently large $n$, there exists a graph $G$ on $n$ vertices with the following properties:
 \begin{enumerate}[label=\abc]
\item $\delta(G) \ge \big(\delta_\chi(H')-\gamma\big)n$.
\item There exists a set $X \subset V(G)$, with $|X| = K$, such that $\chi\big(G[X]\big) > C$.
\item $\delta_\chi(H^*) < \delta_\chi(H')$ for every subgraph of $H^* \subset G$ with at most $s$ vertices.
\end{enumerate}
\end{lemma}

Given $H$, let $H'$ be a graph which minimises $\delta_\chi(H')$ such that there exists a homomorphism from $H$ to $H'$. We claim that $\delta_\chi^*(H) \ge \delta_\chi(H')$, i.e., that for every $\gamma > 0$ and $C > 0$, there exists $\eps > 0$ and infinitely many $H$-free graphs $G$ with $\delta(G) \ge \big(\delta_\chi(H') - 2\gamma\big)v(G)$ which cannot be made $C$-colourable by removing at most $\eps n^2$ edges.  

To prove this, first let $G'$ be the graph (on $n$ vertices) given by Lemma~\ref{lem:construction:allsametime} with inputs $H'$, $s = v(H)$, $C$ and $\gamma$. We construct a graph $G$ by blowing up each vertex of $X$ to size $\mu n$ for some constant $\mu > 0$. Now, if $\mu < \gamma / K$, then 
$$\delta(G) \, \ge \, \delta(G') \, \ge \, \big(\delta_\chi(H) - \gamma\big)n \, \ge \, \big(\delta_\chi(H) - 2\gamma\big) v(G),$$
since $v(G) \le \big( 1 + \mu K \big) n$, and if $\mu \ge \sqrt{\eps}$ then the chromatic number of $G$ cannot be decreased by removing fewer than $\mu^2 n^2 \ge \eps n^2$ edges. Thus it only remains to show that $G$ is $H$-free. 

Suppose that $G$ is not $H$-free, and fix a copy of $H$ in $G$. We can construct a subgraph $H^* \subset G'$ by taking all vertices of this copy of $H$ lying outside the blow-up of $X$ in $G$, and each vertex of $X$ in $G'$ whose blow-up in $G$ contains one or more vertices of $H$. Note that $H$ is homomorphic to $H^*$, by construction. But $H^*$ has at most $v(H)$ vertices, and therefore $\delta_\chi(H^*) < \delta_\chi(H')$, by Lemma~\ref{lem:construction:allsametime}$(c)$, which contradicts our choice of $H'$. Hence $G$ is $H$-free, and this completes the proof of Theorem~\ref{thm:deltachistar}.

\ifarxiv{\appendix
\section{The proof of Lemma~\ref{lem:magic}}
\label{app:magic}

Before beginning, we should stress that the proof consists only of small and obvious modifications to the argument in~\cite{ChromThresh}: in fact, the only changes are a modification to the statement and proof of Proposition~\ref{3-to-r-rich} below, and the proof of Lemma~\ref{lem:magic} itself which is essentially contained in the proof of `Theorem 34' there. Most of this appendix is copied unchanged from~\cite{ChromThresh}, and it exists only to facilitate the sceptical reader's verification of Lemma~\ref{lem:magic}.

\begin{definition}[Definition~19 of~\cite{ChromThresh}, Modified Zykov graphs]\label{def:Zykov}
Let $T_1,\ldots,T_\ell$ be (disjoint) trees, and let $T_j$ have bipartition $A_j
\dcup B_j$. We define~$Z_\ell(T_1,\ldots,T_\ell)$ to be the graph on
vertex set $$V\big( Z_\ell(T_1,\ldots,T_\ell) \big) := \bigg( \bigcup_{j \in
[\ell]} A_j \cup B_j \bigg) \cup \big\{ u_I \colon I \subset [\ell] \big\}$$ and
with edge set
\begin{equation*}
E\big( Z_\ell(T_1,\ldots,T_\ell) \big) \, := \, \bigcup_{j = 1}^\ell
\Bigg(E(T_j) \cup \bigcup_{j \in I \subset [\ell]} K\big( u_I,A_j \big) \cup
\bigcup_{j \not\in I \subset [\ell]} K\big( u_I,B_j \big) \Bigg).
\end{equation*}
For each $r \ge 3$ and $t \in \NATS$, the \emph{modified Zykov graph}
$Z_\ell^{r,t}(T_1,\ldots,T_\ell)$ is the graph obtained from
$Z_\ell(T_1,\ldots,T_\ell)$ by performing the following two operations:
\begin{enumerate}[label=\abc]
  \item Add vertices $W = \{ w_1, \dots, w_{r-3} \}$, and all edges with
    an endpoint in $W$.
  \item Blow up each vertex $u_I$ with $I \subset [\ell]$ and each
    vertex $w_j$ in $W$ to a set $S_I$ or $S'_j$, respectively, of size $t$.
\end{enumerate}
Finally, we shall write $Z_\ell^{r,t}$ for the modified Zykov graph obtained
when each tree $T_i$, $i\in[\ell]$, is a single edge; that is,
$Z_\ell^{r,t} = Z_\ell^{r,t}(e_1,\ldots,e_\ell)$.
\end{definition}

\begin{obs}[Observation~20 of~\cite{ChromThresh}]\label{acy=Zyk}
Let $\chi(H) = r$. Then $H$ is $r$-near-acyclic if and only if there exist trees
$T_1,\ldots,T_\ell$ and $t \in \NATS$  such that $H$ is a subgraph of
$Z_\ell^{r,t}(T_1,\ldots,T_\ell)$.
\end{obs}

It will be convenient for us to provide a compact piece of notation for the
adjacencies in $Z_\ell^{r,t}$. For this purpose, given a graph $G$ and a set $Y\subset
V(G)$, and integers $\ell,t \in \NATS$ and $r \ge 3$, define $\cG^{r,t}_\ell
(Y)$ to be the collection of functions
\[S \,:\, 2^{[\ell]} \cup [r-3] \,\to\, \binom{Y}{t}\,.\]
It is natural to think of $S$ as a family $\{S_I : I \subset [\ell]\} \cup \{
S'_j : j \in [r-3]\}$ of subsets of $Y$ of size $t$. We say that $S \in
\cG_\ell^{r,t}(Y)$ is \emph{proper} if these sets are pairwise disjoint and
$E(G)$ contains all edges $xy$ with $x\in S_I\cup S'_j$ and $y\in S'_{j'}$, whenever $j\neq j'$. We
shall write $\cF_\ell^{r,t}(Y)$ for the collection of proper functions
in $\cG_\ell^{r,t}(Y)$. The idea behind this definition is that we will later
want to consider a vertex set $Y\subset V(G)$ and a family of disjoint subsets
$\{S_I\colon I\subset[\ell]\}\cup\{S'_j\colon j\in[r-3]\}$ of size $t$ in $Y$
that we want to extend to a copy of $Z_\ell^{r,t}$.

For an ordered pair $(x,y)$ of vertices of $G$, a function
$S\in\cF_\ell^{r,t}(Y)$, and $i \in [\ell]$, we write $(x,y) \to_i S$, if
$S'_j \subset N(x,y)$ for every $j \in [r-3]$ and
\[\bigcup_{I \,:\, i \in I} S_I \subset N(x)  \qquad \textup{and} \qquad
\bigcup_{I \,:\, i \not\in I} S_I \subset N(y)\,.\] For an edge $e=xy\in
E(G)$, we write $e\to_i S$ if either $(x,y)\to_i S$ or $(y,x)\to_i S$.
Recall that $\tpl{e}{\ell}$ denotes the $\ell$-tuple $(e_1,\ldots,e_\ell)$, with
$\tpl{e}{0}$ the empty tuple. Define
\[\tpl{e}{\ell} \to S \qquad \Leftrightarrow \qquad e_i
\to_i S \quad \textup{for each $i \in [\ell]$}\,.\] 
Observe that the graph
$Z_\ell^{r,t}$ consists of a set of pairwise disjoint edges $e_1,\ldots,e_\ell$
and an $S \in \cF_\ell^{r,t}(Y)$ such that $\tpl{e}{\ell} \to S$. An advantage
of this notation is that we can write $\tpl{e}{\ell}\to S$ even if the edges in
$\tpl{e}{\ell}$ are not pairwise disjoint.

We shall show how to find a well-structured set of many
copies of $Z_\ell^{r,t}$ inside a graph with high minimum degree and high chromatic
number. The following definition (in which we shall make use of the compact
notation just defined) makes the concept of `well-structured' precise. Given $X \subset V(G)$, we write $E(X)$ for the edge set of $G[X]$, and if $D \subset E(G)$, then $\delta(D)$ denotes the minimum degree of the
graph $G[D]$.

In the following definition, the reader should think of the sets $D(\tpl{e}{j})$ as a `hierarchy' of graphs: we have a different graph $D(\tpl{e}{j+1})$ associated to each edge of $D(\tpl{e}{j})$. Note that if the vector $\tpl{e}{j}$ and the edge $e_{j+1}$ are contained in the same statement, then $\tpl{e}{j+1}$ is assumed to be their concatenation.

\begin{definition}[Definition~21 of~\cite{ChromThresh}, $(C,\alpha)$-rich in copies of $Z_\ell^{r,t}$]\label{def:rich}
Let~$X$ and~$Y$ be disjoint vertex sets in a graph~$G$, let~$C \in \NATS$ and
$\alpha > 0$, and let $s := (2^\ell+ r - 3) t$. We say that $(X,Y)$ is
\emph{$(C,\alpha)$-rich} in copies of $Z_\ell^{r,t}$ if
\begin{align*}
& \exists \, D=D(\tpl{e}{0})\subset E(X) \; \forall  \, e_1\in D \; \exists \,
D(\tpl{e}{1})\subset E(X) \; \forall \, e_2 \in D(\tpl{e}{1}) \quad \dots \\ &
\hspace{2cm} \dots \quad \forall \, e_{\ell-1} \in D(\tpl{e}{\ell-2}) \;
\exists \, D(\tpl{e}{\ell-1}) \subset E(X) \; \forall \, e_\ell \in
D(\tpl{e}{\ell-1})
\end{align*}
the following properties hold:
\begin{enumerate}[label=\abc]
  \item\label{def:rich:a} $\delta(D), \delta\big(D(\tpl{e}{1})\big), \dots,
  \delta\big(D(\tpl{e}{\ell-1})\big) > C$, and 
  \item\label{def:rich:b} $\left|
  \big\{ S \in \cF_\ell^{r,t}(Y) \,:\, \tpl{e}{\ell} \to S \big\} \right|  \ge
  \alpha |Y|^s$.
\end{enumerate}
\end{definition}

The point of this definition is that since $Z_\ell(3,t)$ is quite a simple graph, we require only quite weak properties of $(X,Y)$ in order to show that $(X,Y)$ is $(C,\alpha)$-rich in copies of $Z_\ell^{3,t}$. This is the content of the following proposition, whose proof is a `paired VC-dimension' argument.

\begin{prop}[Proposition~26 of~\cite{ChromThresh}]\label{lem:VC}
  For every $\ell,t \in \NATS$ and $d > 0$, there exists $\alpha > 0$ such
  that, for every $C \in \NATS$, there exists $C' \in \NATS$ such that the following holds.
  Let $G$ be a graph and let $X$ and $Y$ be disjoint subsets of $V(G)$, such that
  $|N(x) \cap Y| \ge d|Y|$ for every $x \in X$.
  
  Then either $\chi\big( G[X] \big) \le C'$, or $(X,Y)$ is $(C,\alpha)$-rich in
  copies of $Z_\ell^{3,t}$.
\end{prop}

Unfortunately, it is not easy to work with the concept of $(C,\alpha)$-richness. The reason for this is the quantifier alternation in the definition: we would like to construct a tree $T_1$ in $D$, but different edges $e_1$ of $D$ may give us entirely different sets $D(\tpl{e}{1})$, and we need to construct a tree $T_2$ which is in $D(\tpl{e}{1})$ for each $e_1\in T_1$. The next definition eliminates this quantifier alternation. We write $\ol{d}(E)$ for the average degree of $E$, i.e.\ $2|E|$ divided by the number of vertices contained in some edge of $E$.

\begin{definition}[Definition~23 of~\cite{ChromThresh}, Good function, $(C,\alpha)$-dense]\label{def:good}
  A function $S \in \cF_\ell^{r,t}(Y)$ is \emph{$(r,\ell,t,C,\alpha)$-good}
  for  $(X,Y)$ if there exist sets 
  \[E_{1},\ldots,E_\ell \subset E(X), \quad \textup{with} \quad
  \ol{d}(E_j) \ge 2^{-\ell} \alpha C \quad \textup{ for each $1 \le j
    \le \ell$\,,}\]
    such that for every $e_{1} \in E_{1}, \ldots, e_\ell
  \in E_\ell$, we have $\tpl{e}{\ell} \to S$.

  The pair $(X,Y)$ is \emph{$(C,\alpha)$-dense} in copies of $Z_\ell^{r,t}$
  if there exist at least $2^{-\ell} \alpha |Y|^s$ families $S \in \cF(Y)$
  which are $(r,\ell,t,C,\alpha)$-good for $(X,Y)$. 
\end{definition}

To go with this definition we have the following lemma, whose proof is an inductive double counting argument, which converts richness into the more useable denseness.

\begin{lemma}[Lemma~24 of~\cite{ChromThresh}]\label{q-induc}
  If $(X,Y)$ is $(C,\alpha)$-rich in copies of
  $Z_\ell^{r,t}$, then $(X,Y)$ is $(C,\alpha)$-dense in copies of
  $Z_\ell^{r,t}$.
\end{lemma}

Since we only obtain richness in copies of $Z_\ell^{3,t}$ from Proposition~\ref{lem:VC}, but the conclusion of Lemma~\ref{lem:magic} which we want to prove speaks of $Z_\ell^{r,t}$ for all $r\ge 3$, if $r\ge 4$ we need at some point to `upgrade' the structure we have. The lemma which permits us to do this is the following. It is a small modification of Proposition 36 in~\cite{ChromThresh}.

\begin{prop}\label{3-to-r-rich}
  For every $r>3$, $\ell,t \in \NATS$ and $d,\gamma > 0$ there exist
  $\ell^*,t^*\in\NATS$ such that for every $\alpha>0$ and $C\in\NATS$, there
  exist $\eps_1>0$ and $C^*\in\NATS$, such that for every $0<\eps<\eps_1$ the
  following holds.
  
  Let $G$ be a graph, and let $X$, $Y$ and $Z_1,\ldots,Z_{r-3}$ be disjoint
  subsets of $V(G)$, with $|Y| = |Z_j|$ for each $j \in [r-3]$. Let
  $Z:=Z_1\cup\cdots\cup Z_{r-3}$. Suppose that
  $(Y,Z_j)$ and $(Z_i,Z_j)$ are $(\eps,d)$-regular for each $i \ne j$, and that
  for each $e\in G[X]$ and $j \in [r-3]$, the edge $e$ has at least $\gamma |Z_j|$ common neighbours in $Z_j$.
  
  If $(X,Y)$ is $(C^*,\alpha)$-dense in copies of $Z_{\ell^*}^{3,t^*}$, then
  there is some $S\in\cF_\ell^{r,t}(Y\cup Z)$ such that
  $S$ is $(r,\ell,t,C,\alpha)$-good for $(X,Y\cup Z)$.
\end{prop}

The change here, as compared to~\cite{ChromThresh}, is that we insist that each edge $e\in G[X]$ has common neighbourhood $\gamma |Z_j|$ in each $Z_j$, as opposed to that each vertex of $X$ has neighbourhood at least $\big(\tfrac12+\gamma\big)|Z_j|$ in $Z_j$, which latter obviously implies the former. The change in the proof is similarly trivial: the only use made of the stronger condition in~\cite{ChromThresh} is to imply the weaker condition. Nevertheless, we give the details.

For the proof, we combine an application of the
Counting Lemma and two uses of the pigeonhole principle. As a preparation
for these steps we need to show that there exists a family
$S^*\in\cF_{\ell^*}^{3,t^*}$ which is $(3,\ell^*,t^*,C^*,\alpha)$-good for
$(X,Y)$ and `well-behaved' in the following sense. For each of the sets
$S^*_I\subset Y$ given by~$S^*_I$ only a small positive fraction of the
$(r-3)t$-element sets in~$Z$ has a common neighbourhood in $S^*_I$ of less
than~$t$ vertices. To this end we shall use the following lemma.

Recall that for a set
$T$ of vertices in a graph $G$, we write
\[N(T)\colon = \bigcap_{x\in T} N(x)\,.\]

\begin{lemma}[Lemma~37 of~\cite{ChromThresh}]\label{countST}
  For all $r,t \in \NATS$ and $\mu,d > 0$, there exist $t^* =
  t^*(r,t,\mu,d) \in \NATS$ and $\eps_0 = \eps_0(r,t,\mu,d) > 0$ such that for
  all $0<\eps<\eps_0$ the following holds.
  
  Let $G$ be a graph, and suppose that $Y$ and $Z_1,\ldots,Z_{r-3}$ are disjoint
  subsets of $V(G)$ such that $(Y,Z_j)$ is $(\eps,d)$-regular for each $j \in
  [r-3]$. Let $Z := Z_1 \cup \ldots \cup Z_{r-3}$, and define
  \[\cB(S) := \Big\{ T \in \binom{Z}{(r-3)t} \,:\, |N(T) \cap S | < t \Big\}\]
  for each $S \subset Y$. Then we have
  \[\cS:=\Big\{ S \in \binom{Y}{t^*} \,:\, |\cB(S)| \ge \mu
  |Z|^{(r-3)t} \Big\}  \; \le \; \sqrt{\eps} |Y|^{t^*}\,.\]
\end{lemma}

We shall now prove Proposition~\ref{3-to-r-rich}.

\begin{proof}[Proof of Proposition~\ref{3-to-r-rich}]
  We start by defining the constants.
  Given $r>3$, $\ell,t \in \NATS$ and $\gamma,d  > 0$, we set
  \begin{equation}\label{eq:3tor:setalphaell}
    \mu:=\frac{\gamma^{(r-3)t}}{8\big((r-3)t\big)!(r-3)^{(r-3)t}}
    \Big(\frac{d}{2}\Big)^{\binom{r-3}{2}t^2} \quad\text{and}\quad
    \ell^*:=\frac{\ell}{2\mu}\,.
  \end{equation}
  Let $t^*$ and $\eps_0$ be given by Lemma~\ref{countST} with input
  $r, t, \mu':=2^{-\ell^*}\mu, d$. Given $\alpha>0$ and $C$, we choose
  \begin{equation}\label{eq:3tor:setepsCstar}
    \eps_1:=\min\Big(\frac{\alpha^2}{2^{4\ell^*+1}},
    \frac{d\gamma}{4(\gamma+1)(r-3)t},\eps_0\Big)\quad\text{and}\quad
    C^*:=\frac{2^{\ell^*}C}{\alpha\mu}\,.
  \end{equation}
  
  Now let $0<\eps<\eps_1$, let $G$ be a graph, and let $X$, $Y$ and
  $Z_1,\ldots,Z_{r-3}$ be disjoint subsets of $V(G)$ as described in the
  statement, so in particular, $(X,Y)$ is $(C^*,\alpha)$-dense in copies of
  $Z_{\ell^*}^{3,t^*}$. The goal is to
  show that there exists $S\in\cF_\ell^{r,t}(Y\cup Z)$ such that $S$ is
  $(r,\ell,t,C,\alpha)$-good for $(X,Y\cup Z)$.
  
  Our first step is to show that there is a `well-behaved' function
  $S^*\in\cF_{\ell^*}^{3,t^*}(Y)$.
  
  \begin{claim}\label{clm:wellbvd} There is a function
  $S^*\in\cF_{\ell^*}^{3,t^*}(Y)$ which is $(3,\ell^*,t^*,C^*,\alpha)$-good for
  $(X,Y)$ and has the property that for every $I\subset[\ell^*]$, the set
  \[\cB(S^*_I)=\Big\{T\in\binom{Z}{(r-3)t}\colon \big|N(T)\cap S^*_I\big|\le
  t\Big\}\]
  in $\binom{Z}{(r-3)t}$ has size at most $2^{-\ell^*}\mu |Z|^{(r-3)t}$.
  \end{claim}
  
  \begin{claimproof}[Proof of Claim~\ref{clm:wellbvd}]
    By Lemma~\ref{countST} (with input $r,t,\mu'=2^{-\ell^*}\mu,d$), the
    total number of `bad' $t^*$-subsets $S'$ of $Y$, i.e., those for which
    $\cB(S')\ge 2^{-\ell^*}\mu|Z|^{(r-3)t}$, is at most
    $\sqrt{\eps}|Y|^{t^*}$. Let $\cS$ be the set of functions $S^*$ in
    $\cF_{\ell^*}^{3,t^*}(Y)$ which do \emph{not} have the property that
    for every $I\subset[\ell^*]$ we have
    $\cB(S^*_I)<2^{-\ell^*}\mu|Z|^{(r-3)t}$. We can obtain any function
    $S^*$ in $\cS$ by taking a set $I\subset [\ell^*]$ and one of the at
    most $\sqrt{\eps}|Y|^{t^*}$ `bad' $t^*$-sets to be $S^*_I$, and
    choosing the $2^{\ell^*}-1$ remaining sets of $S^*$ in any way from
    $\binom{Y}{t^*}$. It follows that
    \[|\cS|\le
    2^{\ell^*}\sqrt{\eps}|Y|^{t^*}|Y|^{(2^{\ell^*}-1)t^*}=2^{\ell^*}\sqrt{\eps}|Y|^{2^{\ell^*}t^*}\,.\]
    
    Since $(X,Y)$ is $(C^*,\alpha)$-dense in copies of $Z_{\ell^*}^{3,t^*}$,
    there are at least $2^{-\ell^*}\alpha|Y|^{2^{\ell^*}t^*}$ functions in
    $\cF_{\ell^*}^{3,t^*}(Y)$ which are $(3,\ell^*,t^*,C^*,\alpha)$-good for
    $(X,Y)$. Since by~\eqref{eq:3tor:setepsCstar} we have
    $2^{-\ell^*}\alpha> 2^{\ell^*}\sqrt{\eps}$, at
    least one of these functions is not in $\cS$, as required.
  \end{claimproof}
  
  For the remainder of the proof, $S^*$ will be a fixed function satisfying the
  conclusion of Claim~\ref{clm:wellbvd}. Since $S^*$ is
  $(3,\ell^*,t^*,C^*,\alpha)$-good for $(X,Y)$, there exist sets
  \[E^*_1,\ldots,E^*_{\ell^*} \subset E(X), \quad \text{with}  \quad
  \ol{d}(E^*_j) \ge 2^{-\ell^*} \alpha C^* \quad \text{ for each } \quad 1 \le j
  \le \ell^*\,,\] such that for every $e_{1} \in E^*_{1}, \ldots, e_{\ell^*} \in
  E^*_{\ell^*}$, we have $\tpl{e}{\ell^*} \to S^*$.
  
  Our next claim comprises two applications of the pigeonhole
  principle to find a copy of $K_{r-3}(t)$ in $Z$.
  
  \begin{claim}\label{3-to-r:claim2}
    There exists a copy $T$ of $K_{r-3}(t)$ with $t$ vertices in $Z_j$ for each
    $j \in [r-3]$, and a set $L \subset [\ell^*]$ of size $|L| = \ell$ such that:
    \begin{enumerate}[label=\rom]
      \item\label{3-to-r:a} $|N(T) \cap S^*_I | \ge t$  for every $I \subset
      [\ell^*]$,
      \item\label{3-to-r:b} $N(T)$ contains at least $\mu
      |E^*_j|$ edges of $E^*_j$, for each $j \in L$.
    \end{enumerate}
  \end{claim}
  
  \begin{claimproof}[Proof of Claim~\ref{3-to-r:claim2}]
    By assumption, each edge $e \in E^*_1 \cup
    \ldots \cup E^*_{\ell^*}$ has at least $\gamma |Z_j|$ common neighbours in
    $Z_j$. By the Slicing Lemma, the common neighbours of $e$ in $Z_i$ and
    $Z_j$ form an $(\eps/\gamma,d)$-regular pair for each $1\le i<j\le
    r-3$. By~\eqref{eq:3tor:setepsCstar} we have
    $d-\eps-(r-3)t\eps/\gamma>d/2$.
    Hence, applying the Counting Lemma with~$\eps$ replaced by $\eps/\gamma$ to the graph $H=K_{r-3}(t)$, it follows
    that there are at least 
    \begin{multline*}
        \frac1{\Aut(H)}\Big(d-\frac\eps\gamma v(H)\Big)^{e(H)}\Big(\frac{\gamma|Z|}{r-3}\Big)^{v(H)}
        \\ \ge\frac1{\big((r-3)t\big)!}\Big(\frac
        d2\Big)^{\binom{r-3}{2}t^2}\Big(\frac{\gamma|Z|}{r-3}\Big)^{(r-3)t}
        \geByRef{eq:3tor:setalphaell} 8\mu |Z|^{(r-3)t}
    \end{multline*}
    copies of $K_{r-3}(t)$ in
    $N(e) \cap Z$, each with $t$ vertices in each $Z_j$.
   
    There are therefore, for each $j \in [\ell^*]$, at least $8\mu
    |Z|^{(r-3)t} |E^*_j|$ pairs $(e,T)$, where $e \in E^*_j$ and $T$ is a copy
    of $K_{r-3}(t)$ as described, such that $T \subset N(e)$, or equivalently $e
    \subset N(T)$. Since we have
    \[8\mu
    |Z|^{(r-3)t} |E^*_j|=4 \mu |Z|^{(r-3)t}|E^*_j|+4 \mu
    |E^*_j||Z|^{(r-3)t}\,,\]
    by the pigeonhole principle, it follows that
    there are at least $4 \mu |Z|^{(r-3)t}$ copies of $K_{r-3}(t)$ in~$Z$ each of which
    has at least $4 \mu |E^*_j|$ edges of $E^*_j$ in its common
    neighbourhood. Let us denote by $\cT_j$ the collection of such copies of
    $K_{r-3}(t)$. For a copy $T$ of $K_{r-3}(t)$, let $L(T) = \big\{ j : T
    \in \cT_j \big\}$.
    
    We claim that there is a set $\cT$ containing at least $2\mu
    |Z|^{(r-3)t}$ copies $T$ of $K_{r-3}(t)$ in $Z$, each with $|L(T)|\ge\ell$.
    Indeed, this follows once again by the pigeonhole principle, since there are at least
    \[\ell^*\cdot 4 \mu |Z|^{(r-3)t} \,\eqByRef{eq:3tor:setalphaell}\, \ell
    |Z|^{(r-3)t}+\ell^*\cdot 2\mu|Z|^{(r-3)t}\]
    pairs $(T,j)$ with $T \in \cT_j$.
    
    Now, recall that $S^*$ satisfies the conclusion of Claim~\ref{clm:wellbvd},
    i.e., for each $I\subset[\ell^*]$, there are at most
    $2^{-\ell^*}\mu|Z|^{(r-3)t}$ sets $T \in \binom{Z}{(r-3)t}$ such
    that $|N(T) \cap S^*_I | \le t$. Since $|\cT|\ge 2\mu|Z|^{(r-3)t}$, there
    is a copy $T$ of $K_{r-3}(t)\in\cT$ such that for each $I\subset[\ell^*]$,
    we have $|N(T)\cap S^*_I|\ge t$. If we let~$L$ be any subset of $L(T)$ of size
    $\ell$, then $T$ and $L$ satisfy the conclusions of the claim.
  \end{claimproof}
  
  Let $T$ and $L$ be as given by Claim~\ref{3-to-r:claim2} and for each $j \in L$ let
  $E_j\subset X$ be a set of $\mu|E^*_j|$ edges of $E^*_j$ contained in $N(T)$ as
  promised by Claim~\ref{3-to-r:claim2}\ref{3-to-r:b}.
  We construct a
  function $S \in \cF_\ell^{r,t}(Y)$ by choosing, for each $I \subset L$, a
  subset $S_I \subset S^*_I$ of size~$t$ in $N(T)\cap Y$ (which is possible
  by Claim~\ref{3-to-r:claim2}\ref{3-to-r:a}), and letting the sets $S_i$,
  $i\in[r-3]$, be the parts of $T$.
  
  \begin{claim}\label{3-to-r:claim3}
    $S$ is $(r,\ell,t,C,\alpha)$-good for $(X,Y \cup Z)$.
  \end{claim}
  \begin{claimproof}[Proof of Claim~\ref{3-to-r:claim3}]
    Recall that $|L| = \ell$, and assume without loss of generality that $L
    = \{1,\ldots,\ell\}$. 
    By the choice of~$T$ and the definition of the sets $S_I$ with $I
    \subset L$ and the sets~$S_i$ with
    $i\in[r-3]$, we have that $S_i$ is completely adjacent to each $S_{i'}$
    with $i\neq i'$, to each~$S_I$, and to each edge $e\in \bigcup_{j\in L} E_j$.
    Since $\tpl{e}{\ell^*} \to S^*$ for each
    $\tpl{e}{\ell^*} \in E^*_1 \times \ldots \times E^*_{\ell^*}$, it
    follows that $\tpl{e}{\ell} \to S$ for each $\tpl{e}{\ell} \in E_1
    \times \ldots \times E_{\ell}$. Finally, for each $j\in L$, since
    $|E_j|\ge\mu|E^*_j|$, we have
    \[\ol{d}(E_j)\ge \mu\ol{d}(E^*_j)\ge \mu2^{-\ell^*}\alpha
    C^*\eqByRef{eq:3tor:setepsCstar} C\,,\]
    as required.
  \end{claimproof}
  
  Thus there exists a function $S \in \cF_\ell^{r,t}(Y)$ which is
  $(r,\ell,t,C,\alpha)$-good for $(X,Y \cup Z)$, as required.
\end{proof}

Our final lemma states that the existence of a good function as provided by the previous proposition indeed implies the existence of the desired $r$-near-acyclic graph.

\begin{lemma}[Lemma~25 of~\cite{ChromThresh}]\label{lem:denseembed} 
Let $X$ and $Y$ be disjoint vertex sets
  in $G$. Given $r,\ell,t\in\NATS$, $\alpha>0$, and trees $T_1,\ldots,T_\ell$,
  if $C\ge 2^{\ell+3}\alpha^{-1}\sum_{i=1}^\ell |T_i|$ and
  $S\in\cF_\ell^{r,t}(Y)$ is $(r,\ell,t,C,\alpha)$-good for $(X,Y)$, then
  $Z_\ell^{r,t}(T_1,\ldots,T_\ell)\subset G$.
\end{lemma}

We can now complete the proof of Lemma~\ref{lem:magic}.

\begin{proof}[Proof of Lemma~\ref{lem:magic}]
  Let $H$ be an $r$-near-acyclic graph,
  with $r \ge 3$, and let $\gamma > 0$. We set $d=\gamma=\gamma$. Because $H$ is $r$-near-acyclic, by
  Observation~\ref{acy=Zyk} there exist trees $T_1,\ldots,T_\ell$ and a
  number $t\in\NATS$ such that $H\subset
  Z_\ell^{r,t}(T_1,\ldots,T_\ell)$. We now set constants as follows. First,
  we choose $d=\gamma$. Given $r$, $\ell$, $t$, $d$ and $\gamma$,
  Proposition~\ref{3-to-r-rich} returns integers $\ell^*$ and $t^*$.  Now
  Proposition~\ref{lem:VC}, with input $\ell^*,t^*$ and $d$, returns
  $\alpha>0$. Next, consistent with Lemma~\ref{lem:denseembed} we set
  $C:=2^{\ell+3}\alpha^{-1}\sum_{i=1}^\ell |T_i|$. Feeding $\alpha$ and $C$
  into Proposition~\ref{3-to-r-rich} yields $\eps>0$ and $C^*$. Putting
  $C^*$ into Proposition~\ref{lem:VC} yields a constant $D=C'$.
  
  Now suppose we are given a graph $G$ containing pairwise disjoint vertex sets $X$, $Y$, $Z_1,\dots,Z_{r-2}$ such that $G[X]$ has chromatic number at least $D+1$, and such that $|Y|=|Z_1|=\dots=|Z_{r-2}|=m$, such that each vertex of $X$ has at least $\gamma m$ neighbours in $Y$ and each edge of $G[X]$ has at least $\gamma m$ common neighbour in each $Z_1,\ldots,Z_{r-2}$, and such that each pair from $Y,Z_1,\dots,Z_{r-2}$ forms an $\eps$-regular pair in $G$ of density at least $\gamma$. Our aim is to prove $H\subseteq G$.

  We apply Proposition~\ref{lem:VC}, with input $\ell^*$, $t^*$, $d$ and
  $C^*$, to $(X,Y)$. By assumption, we have $|N(x)\cap
  Y|\ge d|Y|$ for each $x\in X$. Since $\chi\big(G[X]\big)\ge D+1$, so $(X,Y)$ is $(C^*,\alpha)$-rich in copies of
  $Z_{\ell^*}^{3,t^*}$.
  
  By Lemma~\ref{q-induc} the pair $(X,Y)$ is
  $(C^*,\alpha)$-dense in copies of $Z_{\ell^*}^{3,t^*}$. We now apply
  Proposition~\ref{3-to-r-rich}, with input $r$, $\ell$, $t$, $d$,
  $\gamma$, $\alpha$, $C$, and~$\eps$
  to $X,Y,Z_1,\ldots,Z_{r-3}$. By assumption, each edge of $G[X]$ has at least $\gamma |Z_i|$ common neighbours in each $Z_i$, and any pair of $Y,Z_1,\ldots,Z_{r-3}$ is
  $(\eps,d)$-regular. By Proposition~\ref{3-to-r-rich}, there exists
  a function $S\in\cF_\ell^{r,t}(Y\cup Z_1\cup\cdots\cup Z_{r-3})$ which is
  $(r,\ell,t,C,\alpha)$-good for $(X,Y\cup Z_1\cup\cdots\cup Z_{r-3})$.
  Finally, we apply Lemma~\ref{lem:denseembed}, with input
  $r,\ell,t,\alpha$ and $T_1,\ldots,T_\ell$, to $X$ and $Y\cup Z_1\cup\cdots\cup
  Z_{r-3}$ with the function $S$. By the definition of~$C$, this lemma gives that
  $Z_\ell^{r,t}(T_1,\ldots,T_\ell)$ is contained in $G$, and so $H\subseteq G$ as desired.
\end{proof}
}\fi


\bibliographystyle{amsplain_yk} 
\bibliography{ChromaticThreshold}

\end{document}